\documentclass[12pt]{amsart}
\usepackage{latexsym}
\usepackage{amsmath}
\usepackage{amsfonts}
\usepackage{amsthm}
\usepackage{amssymb}
\usepackage{ifthen}
\usepackage{enumerate}
\usepackage[T1]{fontenc}
\usepackage[mathscr]{eucal}
\newcommand*{\rom}[1]{\expandafter\@slowromancap\romannumeral #1@}

\DeclareFontFamily{U}{mathx}{}
\DeclareFontShape{U}{mathx}{m}{n}{<-> mathx10}{}
\DeclareSymbolFont{mathx}{U}{mathx}{m}{n}
\DeclareMathAccent{\widehat}{0}{mathx}{"70}
\DeclareMathAccent{\widecheck}{0}{mathx}{"71}
\def\eq#1{{\rm(\ref{E#1})}}
\def\Eq#1#2{\ifthenelse{\equal{#1}{*}}
  {\begin{equation*}\begin{aligned}#2\end{aligned}\end{equation*}}
  {\begin{equation}\begin{aligned}\label{E#1}#2\end{aligned}\end{equation}}}

\newcounter{allenv}[section]

\newtheorem{thm}[allenv]{Theorem}
\newtheorem{prop}[allenv]{Proposition}
\newtheorem{coro}[allenv]{Corollary}
\newtheorem{lemm}[allenv]{Lemma}

\theoremstyle{remark}
\newtheorem{remark}[allenv]{Remark}
\newtheorem{exmp}[allenv]{Example}

\theoremstyle{definition}

\newcommand{\NN}{\mathbb{N}}

\newcommand{\RR}{\mathbb{R}}

\newcommand{\supp}{\mbox{\rm supp}}

\newcommand{\map}{\to}

\def\diam{\mbox{\rm diam}}
\def\conv{\mbox{\rm conv}}
\author[Tibor Kiss]{Tibor Kiss}
\author[P\'eter T\'oth]{P\'eter T\'oth}

\title[Improved regularity for a composite functional equation]{Improved regularity for a composite functional equation stemming from the theory of means} 

\address{Institute of Mathematics,
University of Debrecen,
4002 Debrecen, Pf.~400, Hungary}
\email{toth.peter@science.unideb.hu}
\email{kiss.tibor@science.unideb.hu}

\keywords{regularity improvement, composite functional equation, 
additive function, affine function} 

\subjclass[2020]{39B22, 39B72, 26B05}

\thanks{Partial support for the first author's research was provided by NKFIH Grant K-134191 and by funding from the HUN-REN Hungarian Research Network. The research of the second author has been supported by 
the EK\"OP-24-0 University Research Scholarship Program of 
the Ministry for Culture and Innovation 
from the source of the 
National Research, Development and Innovation Fund; 
by the PhD Excellence Scholarship from the 
Count Istv\'an Tisza Foundation 
for the University of Debrecen; 
and by the University of Debrecen Program 
for Scientific Publication}

\def\eq#1{{\rm(\ref{E#1})}}

\def\Eq#1#2{\ifthenelse{\equal{#1}{*}}
	{\begin{equation*}\begin{aligned}[]#2\end{aligned}\end{equation*}}
	{\begin{equation}\begin{aligned}\label{E#1}#2\end{aligned}\end{equation}}}

\begin{document}

\begin{abstract}
In this paper we describe the solutions of 
the functional equation 
\begin{equation*}
F\Big(\frac{x+y}2\Big)+f_1(x)+f_2(y)= 
G \big(g_1(x)+g_2(y)) 
\end{equation*} 
defined on an open subinterval of $ \mathbb{R} $. 
Improving previous results we assume differentiability on 
each involved function, eliminate a former condition on 
$ g'_1 $ and $ g'_2 \, $, moreover we determine a brand 
new family of solutions. We also present a particular member 
of this class as an example. 
In order to achieve this, we strengthen known results 
about certain auxiliary functional equations as well. 
\end{abstract}

\maketitle

\section{Introduction}\label{Intro}

This paper serves both as a continuation and a revision of the paper \cite[T. Kiss, 2024]{Kis22}. We describe the solutions of the Pexider Composite Functional Equation
\begin{equation}\label{eq-Invariance-eq}
F\Big(\frac{x+y}2\Big)+f_1(x)+f_2(y)= 
G \big(g_1(x)+g_2(y)) 
\end{equation} 
under mild regularity assumptions on the six unknown functions $F,f_1,f_2,g_1,g_2:I\to\RR$ and $G:g_1(I)+g_2(I)\to\RR$. Given that the equation in question is an equivalent reformulation of the invariance equation of generalized weighted quasi-arithmetic means, the natural conditions for the unknown functions would be continuity and strict monotonicity. Moreover, one may even assume that $g_1$ and $g_2$ are strictly monotone functions in the same sense. This condition also plays an important role in the present paper. The precise manner in which equation \eqref{eq-Invariance-eq} is connected to the invariance equation is detailed in \cite{Kis22}. 

For the sake of completeness, 
we have to mention the paper of 
Matkowski \cite{Mat10}, where class of 
generalized weighted quasi-arithmetic means is introduced. 
The invariance problem for these means covers many 
notable and important invariance problems about various 
types of means. 
Concerning this topic, we would like to turn the 
attention of the interested reader towards the works of 
Matkowski \cite{Mat99}, Jarczyk \cite{JM06, Jar07}, 
Daróczy, Páles \cite{DP01, DP02} and 
Maksa \cite{DMP00}  
A detailed summary about invariance problems of means 
can be found in the survey paper of 
Jarczyk and Jarczyk \cite{JJ18}. 

However, in this paper we focus exclusively 
on the equation \eqref{eq-Invariance-eq} 
itself and not its background.
In the paper \cite{Kis22}, continuously differentiable solutions were described under the additional and rather unnatural condition that the derivatives of the functions $g_1$ and $g_2$ do not vanish anywhere. Such a solution was referred in \cite{Kis22} to as regular. We emphasize that in the present paper, except for the Introduction, we do not intend to use this terminology. The aim of the paper is to improve the results of \cite{Kis22} in three different aspects.
\begin{enumerate}\itemsep=1mm
\item[I.] We are going to eliminate the aforementioned restriction on the derivatives of the functions $g_1$ and $g_2$.
\item[II.] We are going to correct an erroneous statement of \cite{Kis22}, which is the result of a flawed proof. Specifically, the paper \cite{Kis22} contains a Theorem of Alternative claiming that, for any regular solution of \eqref{eq-Invariance-eq}, the function $F$ is either affine over its entire domain or there does not exist any subinterval on which it is 
affine. The third possibility, namely that $F$ is affine only on certain subinterval, was erroneously excluded. In this paper, we also describe these solutions.
\item[III.] We are going to weaken the assumption of continuous differentiability to mere differentiability.
\end{enumerate}

To formulate our results in a more simple and compact form, we will need some notation and conventions. For example, we shall say that $(F,f_k,G,g_k)$ solves equation \eqref{eq-Invariance-eq} on the subinterval $J\subseteq I$ if, for the functions $F:I\to\RR$, $f_1:I\to\RR$, $f_2:I\to\RR$, $g_1:I\to\RR$, $g_2:I\to\RR$, and $G:g_1(I)+g_2(I)\to\RR$, the equation \eqref{eq-Invariance-eq} holds true for all $x,y\in J$.

This principle will apply generally. More precisely, whenever the variable $k$ appears as a subscript, it should always be understood that the given condition or statement holds for both cases $k = 1, 2$. However, in some instances, we will explicitly state the range of the index $k$.

In the sequel, we will divide the family of differentiable solutions of equation \eqref{eq-Invariance-eq} into three pairwise disjoint subfamilies. 

In Section \ref{A}, we describe those solutions for which the function $F$ is affine on the entire domain $I$. It then turns out that $G$ is also affine, the functions $g_1$ and $g_2$ can be chosen freely, and moreover, $f_k$ can be expressed in terms of $F$ and of some affine transformation of $g_k$. We emphasize that, in this case, apart from the continuity and monotonicity of $g_1$ and $g_2$, no other regularity assumptions are used. The main tool of this section is the 
famous theorem of Radó and Baker \cite{RB87}. 

In Section \ref{PA}, we discuss those solutions for which $F$ is affine only on some proper (i.e. non-empty, open, and different from $I$) subinterval of its domain. Such function will be called \emph{partially affine}. This family includes solutions that may be exactly once differentiable. We emphasize that this is a 
completely new class of solutions which has not appeared 
before in the literature of the equation \eqref{eq-Invariance-eq}. 
Therefore, besides the characterization theorem 
of partially affine solutions (namely, Theorem \ref{M3}) 
we are going to present an example of a particular 
solution consisting of exactly once continuously differentiable 
functions. 

The main auxiliary tool of this section is the auxiliary 
functional equation 
\[ 
\varphi\Big(\frac{u+v}2\Big)\big(\psi_1(u)-\psi_2(v)\big)=0
\]
where $ \varphi, \psi_k $ is constructed from 
$ G, g_k $ and their derivatives and they are defined on 
appropriate open intervals. 
In the works \cite{Kis24} and \cite{Tot} the authors 
have described the solutions of this equation under 
weak regularity assumptions such as measurability. Here 
we are going to use the characterization of the solutions 
in the case when $ \varphi $ is a derivative. 

In Section \ref{NA} we turn our attention to solutions 
where $ F $ is nowhere affine. 
Then the derivatives 
$ g'_1 $ and $ g'_2 $ do not vanish at any point. 
Consequently we are able to derive a system of two 
auxiliary functional equations 
\begin{equation}\label{eq-auxeqs}
\begin{cases}
\varphi \left( \frac{x+y}{2} \right) 
\left( \psi_1(x) + \psi_1(y) \right) = 
\varphi(x) \psi_1(x) + \varphi(y) \psi_1(y) \\ 
\varphi \left( \frac{x+y}{2} \right) 
\left( \psi_2(x) - \psi_2(y) \right) = 
\Psi_2(x) - \Psi_2(y) \, 
\end{cases} 
\qquad x,y \in I, 
\end{equation}
where $ \varphi, \psi_1\,, \psi_2 \,, \Psi_2 $ are constructed from 
the derivatives of the functions appearing in equation 
\eqref{eq-Invariance-eq}.  
Our main objective of Section \ref{NA} is to solve this 
system of equations. 
Both of the two equations 
has a rich literature on its own. The first equation 
has been solved by  Daróczy, Maksa and Páles \cite{DMP04}, 
then their result has been improved 
by the first author of this paper jointly with Páles 
\cite{KP18}, weakening the imposed regularity conditions. 
However, in every known result $ \varphi = \frac{1}{2} F' $ 
is assumed to be continuous, 
which is not necessarily fulfilled in our current setting. 
Hence, in order to utilize the previous results, 
in this paper we are going to prove that if 
$ \varphi $ is 
a derivative which is nowhere constant, while $\psi_1$ is 
the difference of the reciprocals of two nonvanishing derivatives, 
then the nontrivial solutions are continuous. 
The proof of this statement 
utilizes classical results from measure theory, 
such as the Lebesgue Density Theorem and Steinhaus' Theorem. 

The second equation has been investigated by many authors, 
such as Balogh, Ibrogimov and Mityagin \cite{BIM16}, moreover 
Łukasik \cite{Luk18}. In their works (higher order) 
differentiability is supposed. 
A significant reduction of regularity assumptions 
is due to the first author and Páles \cite{KP19}, who 
solved the equation assuming merely continuity for $ \varphi $. 

Combining these with our new results concerning the first equation, 
we are able to characterize the solutions of 
\eqref{eq-auxeqs} under reasonably mild regularity. 
We have to emphasize, that in the paper \cite{Kis22} the 
solutions of the separate equations of 
\eqref{eq-auxeqs} are presented, 
but in the setting where each functions in 
\eqref{eq-Invariance-eq} are continuously differentiable. 
Our main advancement in Section \ref{NA} 
of our paper is that 
we are able to deduce the same 
solutions of \eqref{eq-auxeqs} as in \cite{Kis22}, 
supposing merely differentiability for 
$ F, f_k \,, G, g_k $ along with $ g'_1 \cdot g'_2 > 0 $. 
This also implies that in the case when $ F $ is nowhere 
affine then our reduction of the imposed regularity 
conditions does not result in any new solutions. 
In particular, as we are going to establish, if 
$ F $ is nowhere affine then the solutions 
$ \left( F, f_k \,, G, g_k \right) $ 
consist of infinitely many times differentiable 
functions described in \cite{Kis22}.

\subsection{Notations}
Throughout the paper, $I$ stands for a non-empty open interval. For any $a,b\in\RR$ with $ a < b $, the closed and open intervals will be denoted by $\left[a,b\right]$ and $\left]a,b\right[\,$, respectively. 
Accordingly, we will use, for instance, notations such as 
\Eq{*}{\left[ a,b \right[ \, := 
\lbrace t \in \RR : a \leq t < b \rbrace\qquad\text{or}\qquad
\left] a,b \right] \, := 
\lbrace t \in \RR : a < t \leq b \rbrace .} 
For arbitrary sets $ H_1 $ and $ H_2 $ the notation 
$ H_1 \subset H_2 $ means that $ H_1 $ is a subset of 
$ H_2 $ but $ H_1 \neq H_2 \,$.  

Let $A:\RR\to\RR$ be an additive function and $b\in\RR$. We note that by an additive function we mean a solution to the Cauchy Functional Equation.
 Then, for a given function $h:I\to\RR$, we shall write $U\in\mathscr{A}_{h}(A,b)$ if and only if $U$ is a nonempty open subinterval of $I$ and $h(x)=A(x)+b$ whenever $x\in U$. If $h$ is continuous, then, for $A,b\in\RR$, instead of $\mathscr{A}_{h}(A\cdot\mathrm{id},b)$ we simply write $\mathscr{A}_{h}(A,b)$. In this case we also use the function's formula differently, namely, we write $h(x)=Ax+b$ instead of $h(x)=A(x)+b$. The context will always make it clear whether $A$ stands for an additive function or a real number.

Beyond all this, we will write that $U\in\mathscr{A}_{h}^{\max}(A,b)$ if $U\in\mathscr{A}_{h}(A,b)$ and there is no interval $V\in\mathscr{A}_{h}(A,b)$ for which $U\subseteq V$ holds with $U\neq V$. Note that, for any $U_0\in\mathscr{A}_h(A,b)$ there uniquely exists $U\in\mathscr{A}_h^{\max}(A,b)$ such that $U_0\subseteq U$, namely
\Eq{*}{
	U:=\bigcup\{V\in\mathscr{A}_h(A,b)\colon U_0\subseteq V\}.
}

Finally, in certain cases, we would like to indicate that a given function is affine on a specific interval $U$, without needing its explicit representation (i.e., its linear and constant part). In such cases, we simply write $U\in\mathscr{A}_h$ or $U\in\mathscr{A}_h^{\max}$. 

With these notations we may concisely summarize the 
three main part of our paper. 
In Section \ref{A} we deal with the case 
$\mathscr{A}_F^{\max}=\{I\}$, in Section \ref{PA} 
we assume $\mathscr{A}_F \neq \emptyset $ 
but $I\notin\mathscr{A}_F \,$, while 
in Section \ref{NA} we have $ \mathscr{A}_F = \emptyset $. 

\section{Solutions, where $F$ affine on its entire domain}\label{A}

The equation is quite simple to solve in the case when we know that the function $F$ is affine over the entire domain. In this case, the functions $f_1$ and $f_2$ can be expressed in terms of $F$ and of the functions $g_1$ and $g_2$, respectively, and moreover, $G$ also turns out to be affine on its domain. Later on, we will also deal with the case when $F$ is only partially affine. For this, we will need some results from the current section. Therefore, some results will be formulated only for subintervals.

\begin{prop}\label{Aff}
Assume that $F:I\to\RR$ is affine on some nonempty open subinterval $J\subseteq I$ and let $g_1,g_2:I\to\RR$ be continuous and strictly monotone. Then $(F,f_k,G,g_k)$ solves equation \eqref{eq-Invariance-eq} over $J$ if and only if there exist an additive function $B:\RR\to\RR$ and constants $\beta_1,\beta_2\in\RR$ such that
\Eq{*}{
f_k=-\tfrac12F+B\circ g_k+\beta_k
\qquad\text{and}\qquad
G=B+\beta}
hold on $J$ and on $g_1(J)+g_2(J)$, respectively, where $\beta=\beta_1+\beta_2$.
\end{prop}

\begin{proof}
The sufficiency can be verified with straightforward calculations, so we will only prove the necessity. For $k=1,2$, define
\Eq{lk}{\ell_k(x):=F(x)+2f_k(x),\qquad x\in J.}
Expressing the functions $f_1$ and $f_2$ from here and substituting them into our equation, it takes the form
\Eq{*}{
F\Big(\frac{x+y}2\Big)-\frac{F(x)+F(y)}2=G\big(g_1(x)+g_2(y)\big)-\frac{\ell_1(x)+\ell_2(y)}2,\qquad x,y\in J.
}
Utilizing the fact that $F$ is affine on $J$ and that $g_1$ and $g_2$ are invertible functions, the above equation can be rewritten as
\Eq{*}{G(u+v)=\tfrac12\ell_1\circ g_1^{-1}(u)+\tfrac12\ell_2\circ g_2^{-1}(v),\qquad (u,v)\in g_1(J)\times g_2(J).}
The functions $g_1$ and $g_2$ are continuous and strictly monotone, hence $U:=g_1(J)\times g_2(J)$ is an open and connected subset of $\RR\times\RR$. Then, in view of Theorem 1. of \cite[Radó--Baker]{RB87}, there exist an additive function $B:\RR\to\RR$ and constants $\beta_1,\beta_2\in\RR$ such that
\Eq{*}{
\tfrac12\ell_k\circ g_k^{-1}= B+\beta_k
\qquad\text{and}\qquad
G=B+\beta
}
hold on $g_k(J)$ and on $g_1(J)+g_2(J)$, respectively, with $\beta:=\beta_1+\beta_2$. Expressing $\ell_k$ and then using its definition \eq{lk}, we obtain the desired decomposition of $f_1$ and $f_2$ over $J$.
\end{proof}

\begin{thm}\label{Main1}
Let $F:I\to\RR$ be affine and $g_1,g_2:I\to\RR$ be continuous and strictly monotone. Then $(F,f_k,G,g_k)$ solves functional equation \eqref{eq-Invariance-eq} if and only if there exist an additive function $B:\RR\to\RR$ and constants $\beta_1,\beta_2\in\RR$ such that
\Eq{*}{
	f_k=-\tfrac12F+B\circ g_k+\beta_k
	\qquad\text{and}\qquad
	G=B+\beta}
hold on $I$ and on $g_1(I)+g_2(I)$, respectively, where $\beta=\beta_1+\beta_2$.
\end{thm}

\begin{remark}\label{freedom}
	Note that Proposition \ref{Aff} states that whenever $F$ is affine on an interval, the functions $g_1$ and $g_2$ can be chosen freely. This theorem has a counterpart, which asserts that if $g_1$ and $g_2$ are affine functions over an interval, then $F$ can be chosen arbitrarily, provided that $g_1$ and $g_2$ share a common invertible additive part. We now restate this theorem, originally found in \cite{Kis22}, along with its proof.
\end{remark}

\begin{prop}\label{gkc}
Let $J\in\mathscr{A}_{g_1}(D,\delta_1)\cap\mathscr{A}_{g_2}(D,\delta_2)$, where $D:\RR\to\RR$ is an invertible additive function. If $(F,f_k,G,g_k)$ is a solution of equation \eqref{eq-Invariance-eq}, then there exist an additive function $C:\RR\to\RR$ and constants $\gamma_1,\gamma_2\in\RR$ such that
\Eq{*}{
f_k(x)=C(x)+\gamma_k
\qquad\text{and}\qquad
G(u)=F\circ\frac12 D^{-1}(u-\delta)+C\circ D^{-1}(u-\delta)+\gamma
}
hold for all $x\in J$ and $u\in g_1(J)+g_2(J)$, where $\delta=\delta_1+\delta_2$ and $\gamma=\gamma_1+\gamma_2$.
\end{prop}

In the next subsection, we examine this relationship in greater depth. For a given subset $H\subseteq I$, define
\Eq{H-+}{
H^-:=\{x\in I :  x<\inf H\}\qquad\text{and}\qquad
H^+:=\{x\in I :  \sup H<x\}.
}

Obviously, if $H$ is empty, then $H^-=H^+=I$. If $\inf H=\inf I$ or $\sup H=\sup I$, then $H^-=\emptyset$ or $H^+=\emptyset$, respectively. Particularly, if $H=I$, then both of $H^-$ and $H^+$ are empty.

\begin{prop}\label{Gmax}
Let $(F,f_k,G,g_k)$ be a solution of equation \eqref{eq-Invariance-eq} such that $F$ is differentiable, $g_1$ and $g_2$ are continuous and strictly monotone and $U\in\mathscr{A}_F^{\max}$. If $\xi\in I\setminus U$ is a point with the property that there exist additive functions $B,C:\RR\to\RR$ and constants $\gamma,\beta\in\RR$ such that
\Eq{*}{
f_k=B\circ g_k+C+\gamma\qquad\text{and}\qquad
G=B+\beta
}
hold on $J$ and on $g_1(J)+g_2(J)$ with $J:=\conv\big(U\cup\{\xi\}\big)$, then $\xi\in\{\inf U, \sup U\}$.
\end{prop}

\begin{proof}
Indirectly, assume that say $\xi<\inf U$ holds. Then the interval $U^-\cap J$ is not empty. Let $x,y\in U^-\cap J$ be any points. Then, the left hand side of equation \eqref{eq-Invariance-eq} can be written as
\Eq{*}{
F\Big(\frac{x+y}2\Big)+f_1(x)+f_2(y)=F\Big(\frac{x+y}2\Big)+B\big(g_1(x)+g_2(y)\big)+C(x+y)+2\gamma,
}
while the right hand side is of the form
\Eq{*}{
G\big(g_1(x)+g_2(y)\big)=B\big(g_1(x)+g_2(y)\big)+\beta.
}
Simplifying with $B\big(g_1(x)+g_2(y)\big)$ and then putting $x=y$, we obtain that $F(x)=-2C(x)-2\gamma+\beta$, that is, $F$ is affine on $U^-\cap J$. By the differentiability of $F$ we obtain that there exist $A,b\in\RR$ such that $F(x)=Ax+b$ whenever $x\in J$, contradicting the maximality of $U$. This shows that $\xi$ cannot be smaller than $\inf U$. A similar argument yields that $\xi$ cannot be bigger than $\sup U$. Thus the statement follows.
\end{proof}

\section{Solutions, where $F$ is partially affine}\label{PA}

In this section we are going to assume that $F$ is partially affine on its domain. We note that we will continue to rely heavily on the results of Section \ref{A}, namely, we will use the fact that whenever $F$ is affine on a subinterval, the exact form of the functions $f_1$, $f_2$, and $G$ is known there. Furthermore, $g_1$ and $g_2$ can be chosen arbitrarily over such subsets. In the following, we first derive an auxiliary functional equation.

\subsection{The Auxiliary Equation}\label{PA1} The goal of the first part of the current section is to reduce the number of unknown functions in equation \eqref{eq-Invariance-eq}. This will be done in two steps. First, by utilizing the \emph{continuity} and \emph{strict monotonicity} of the functions $g_1$ and $g_2$, we eliminate the functions $f_1$ and $f_2$. In the second step, assuming \emph{differentiability}, we derive an auxiliary functional equation that involves only the functions $g_1$, $g_2$, and $G$, to be more precise, their derivatives and inverses. With these, the number of unknown functions is reduced by half. It is worth noting that our results do not require the continuity of the derivatives.

In order to eliminate the functions $f_1$ and $f_2$, we will need the following sets and some of their properties. The construction itself and the related lemma are not new; they can be found in the paper \cite{Kis22}. However, for easier readability, we will formulate them here as well. The proof, on the other hand, will be omitted. Interested readers can find it in the mentioned paper \cite{Kis22}.

For the functions $g_1,g_2:I\to\RR$, a non-empty subset $H\subseteq I$, and a point $x\in I$, define the restricted preimages
\Eq{*}{
H_k(x):=g_{3-k}^{-1}\big(g_1(H)+g_2(H)-g_k(x)\big)\cap H,\qquad k=1,2.}

Notice that there is a very simple sufficient condition to ensure that none of the sets are empty, namely we have $H_1(x)=H_2(x)=H$, whenever $x\in H$. According to the following lemma, the point $x$ can be chosen as the infimum (or the supremum) of the set $H$, and it is even possible to go slightly beyond this point in such a way that the sets in question remain non-empty, provided that, of course, the infimum (or the supremum) of $H$ is contained in $I$.

\begin{lemm}\label{L6}
Let $g_1$ and $g_2$ be continuous, strictly monotone in the same sense, $H\subseteq I$ be a nonempty open subinterval and set $a:=\inf H$ and $b:=\sup H$. Then, for $k=1,2$, the following statements hold.
\begin{enumerate}\itemsep=1mm
\item\label{L61} If $a\in I$, then $H\subseteq H_k(a)$. If $b\in I$, then $H\subseteq H_k(b)$.
\item\label{L62} If $a\in I$, then there exists $x\in I$ with $x<a$ such that $H_k(x)$ is nonempty. If $b\in I$, then there exists $x\in I$ with $b<x$ such that $H_k(x)$ is nonempty.
\item\label{L63} If $x\in I$ with $x<a$ such that $H_k(x)$ is nonempty, then, for all $u\in[x,a]$, we have 
$H_k(x)\subseteq H_k(u)$. If $x\in I$ with $b<x$ such that $H_k(x)$ is nonempty, then, for all $u\in[b,x]$, we have 
$H_k(x)\subseteq H_k(u)$.
\item\label{L64} If $x\in I$ with $x<a$, then $a<\inf H_k(x)$. In addition, if $H_k(x)$ is not empty, then $\sup 
H_k(x)=b$. If $x\in I$ with $b<x$, then $\sup H_k(x)<b$. In addition, if $H_k(x)$ is not empty, then $\inf H_k(x)=a$.
\end{enumerate}\end{lemm}

\begin{proof}
See the proof of \cite[Lemma 6]{Kis22}.
\end{proof}

We draw attention to the fact that, in view of the above lemma, for a given point $x\in I$, the sets $H_1(x)$ and $H_2(x)$ are either both empty or both non-empty. Now, keeping the conditions of Lemma \ref{L6}, define the set
\Eq{*}{
H_\mathrm{ext}:=\{x\in I :  \text{the intersection }H_1(x)\cap H_2(x)\text{ is not empty}\}.
}

Basically, $H_\mathrm{ext}$ consists of those points $x$ of $I$ for which there exist $u,v\in H$ such that both of $g_1(x)+g_2(u)$ and $g_1(v)+g_2(x)$ belong to $g_1(H)+g_2(H)$. By $H\subseteq H_\mathrm{ext}$, the set $H_\mathrm{ext}$ is not empty. Statement (\ref{L62}) of Lemma \ref{L6} states that the inclusion $H\subseteq H_\mathrm{ext}$ is strict whenever $H\neq I$. Finally, by (\ref{L63}) of Lemma \ref{L6}, $H_\mathrm{ext}$ is a subinterval of $I$, which is obviously open due to the conditions listed in Lemma \ref{L6}.

In summary, it can be stated that under the conditions of Lemma \ref{L6}, $H_\mathrm{ext}$ is a proper extension of the interval $H$ in $I$, provided that $H$ is not equal to $I$.

Beside the above set, we will also use the part of the reflection of the subinterval $H\subseteq I$ onto itself that lies within $I$. So let us define
\Eq{*}{
H_\mathrm{ref}:=(2H-H)\cap I.
}

Note that $H_\mathrm{ref}$ strictly contains $H$ if $H\neq I$ holds true.

Having the above lemma at our disposal, we can now, in a certain sense, extend the form of $f_1$ and $f_2$ given in Proposition \ref{Aff} to intervals on which $F$ is affine only on a subinterval.

\begin{prop}\label{ext}
Let $(F,f_k,G,g_k)$ be a solution of equation \eqref{eq-Invariance-eq} such that $g_1,g_2:I\to\RR$ are continuous and strictly monotone in the same sense. Let further $U\in\mathscr{A}_F(A,b)$. Then there exist an additive function $B:\RR\to\RR$ and constants $\beta_1,\beta_2\in\RR$ such that, for $k=1,2$, we have
\Eq{fk+}{
f_k(x)=-\frac12\big(A(x)+b\big)+B\circ g_k(x)+\beta_k,\qquad x\in U_\mathrm{ext}\cap U_\mathrm{ref}.}
\end{prop}

\begin{proof} If we apply Proposition \ref{Aff} for $J:=U$, we obtain that there exist an additive function $B:\RR\to\RR$ and constants $\beta_1,\beta_2\in\RR$ such that, for $k=1,2$,
\Eq{*}{
f_k=-\frac12(A+b)+B\circ g_k+\beta_k\qquad\text{and}\qquad
G=B+\beta}
are valid on $U$ and on $g_1(U)+g_2(U)$, respectively, with $\beta=\beta_1+\beta_2$. In the sequel, we may and do assume that $\inf U\in I$.
	
Let $x\in U_\mathrm{ext}\cap U_\mathrm{ref}$ be any point. If $x\in U$, then the identities in \eq{fk+} are obviously valid, therefore assume that $x<\inf U$. Then $U_1(x)\cap U_2(x)$ is not empty, furthermore, there exists $v\in U$ such that $\frac12(x+v)\in U$. Let $u\in U_1(x)\cap U_2(x)$ be any point. Note that, by (\ref{L64}) of Lemma \ref{L6}, $[u,\sup U]\subseteq U_1(x)\cap U_2(x)$ follows. Then we assert that for $y:=\max(u,v)$ we have
\Eq{*}{
\frac{x+y}2\in U\qquad\text{and}\qquad y\in U_1(x)\cap U_2(x).
}

Indeed, if $u\leq v$, that is, if $y=v$, then the first inclusion above follows from the definition of $v$. The second inclusion is implied by $[u,\sup U]\subseteq U_1(x)\cap U_2(x)$. If $v<u$, then the second inclusion holds automatically, furthermore, since $U$ is an interval, the validity of the first inclusion is also clear.

Now, apply equation \eqref{eq-Invariance-eq} for the pair $(x,y)$. Using that $F$ is affine on $U$ and the inclusion $y\in U$, the left hand side of (\ref{eq-Invariance-eq}) can be written as
\Eq{*}{
F\Big(\frac{x+y}2\Big)+f_1(x)+f_2(y)
&=\frac12\big(A(x)+A(y)\big)+b+f_1(x)-\frac12F(y)+B\circ g_2(y)+\beta_2\\
&=\frac12A(x)+\frac12b+f_1(x)+B\circ g_2(y)+\beta_2.
}
We can therefore make this latter expression equal to the right-hand side of equation \eqref{eq-Invariance-eq}, namely
\Eq{*}{
\frac12A(x)+\frac12b+f_1(x)+B\circ g_2(y)+\beta_2=G\big(g_1(x)+g_2(y)\big)=B\circ g_1(x)+B\circ g_2(y)+\beta,
}
where in the last step we used the fact that $y\in U_1(x)$. Simplifying the above equality by the term $B\circ g_2(y)$ and using that $\beta_1=\beta-\beta_2$, for the value $f_1(x)$ we get
\Eq{*}{
f_1(x)=-\frac12\big(A(x)+b\big)+B\circ g_1(x)+\beta_1.
}
If $\sup U\in I$, a similar calculation leads to the same result for $x>\sup U$. Consequently, formula \eq{fk+} holds on $U_\mathrm{ext}\cap U_\mathrm{ref}$ for $k=1$. To get the desired form of $f_2$ over $U_\mathrm{ext}\cap U_\mathrm{ref}$, apply equation \eqref{eq-Invariance-eq} for the pair $(y,x)$ and perform the same reasoning.
\end{proof}

We now formulate the auxiliary functional equation, which involves only four unknown functions. This equation will be valid only over intervals of the type $U_\mathrm{ext}\cap U_\mathrm{ref}$, where $U\in\mathscr{A}_F$.

\begin{coro}\label{Aux1}
Let $(F,f_k,G,g_k)$ be a solution of equation \eqref{eq-Invariance-eq} such that $g_1,g_2:I\to\RR$ are continuous and strictly monotone in the same sense and let $U\in\mathscr{A}_F(A,b)$. Then there exist an additive function $B:\RR\to\RR$ and a constant $\beta\in\RR$ such that
\Eq{G+}{
F\Big(\frac{x+y}2\Big)-\frac12\big(A(x)+A(y)\big)-b=G\big(g_1(x)+g_2(y)\big)-B\big(g_1(x)+g_2(y)\big)-\beta
}
holds for all $x,y\in U_\mathrm{ext}\cap U_\mathrm{ref}$.
\end{coro}

\begin{proof}
Apply \eq{fk+} of Proposition \ref{ext} in the equation \eqref{eq-Invariance-eq} restricted to $U_\mathrm{ext}\cap U_\mathrm{ref}$.
\end{proof}

Let us note that both of \eq{fk+} of Proposition \ref{ext} and \eq{G+} of Corollary \ref{Aux1} trivially hold on $U\in\mathscr{A}_F(A,b)$. 

We have now reached the point where we can derive the Auxiliary Functional Equation that will play a central role in the subsequent calculations of this section. As mentioned earlier, differentiability becomes an essential condition at this stage. To make it easier to formulate our further results, from now on, let $U^*$ stand for the intersection $U_\mathrm{ext}\cap U_\mathrm{ref}$, whenever $U\in\mathscr{A}_F$. Clearly, $U^*$ contains such points from $U_\textrm{ext}$ that, when reflected across any element of $U$, still remain within $I$.

\begin{prop}\label{aux}
Let $(F,f_k,G,g_k)$ be a solution of equation \eqref{eq-Invariance-eq} such that each of its coordinate functions is differentiable on its domain, assume that $g_1,g_2:I\to\RR$ are strictly monotone in the same sense, and let $U\in\mathscr{A}_F$ with $U\neq I$. Then there exists a constant $B\in\RR$ such that
\Eq{aux}{
\varphi\Big(\frac{u+v}2\Big)\big(\psi_1(u)-\psi_2(v)\big)=0,\qquad (u,v)\in g_1(U^*)\times g_2(U^*)
}
holds true with
\Eq{phi}{
\varphi(t):=G'(2t)-B\qquad\text{and}\qquad
\psi_k(s):=g_k'\circ g_k^{-1}(s)}
for $t\in\tfrac12\big(g_1(U^*)+g_2(U^*)\big)$ and $s\in g_k(U^*)$, $k=1,2$.
\end{prop}

\begin{proof}
By Corollary \ref{Aux1}, there exist constants $A,B,b,\beta\in\RR$ such that
\Eq{*}{
F\Big(\frac{x+y}2\Big)-\frac12\big(Ax+Ay\big)-b=G\big(g_1(x)+g_2(y)\big)-Bg_1(x)-Bg_2(y)-\beta}
for all $x,y\in U^*$. Differentiate this equation with respect to $x$ and with respect to $y$, separately. Then
\Eq{*}{
\frac12F'\Big(\frac{x+y}2\Big)-\frac12A&=G'\big(g_1(x)+g_2(y)\big)g_1'(x)-Bg_1'(x),\\[1mm]
\frac12F'\Big(\frac{x+y}2\Big)-\frac12A&=G'\big(g_1(x)+g_2(y)\big)g_2'(y)-Bg_2'(y)
}
follows for $x,y\in U^*$. The left-hand sides are identical, so we immediately get that
\Eq{*}{
G'\big(g_1(x)+g_2(y)\big)g_1'(x)-Bg_1'(x)=G'\big(g_1(x)+g_2(y)\big)g_2'(y)-Bg_2'(y),\qquad x,y\in U^*.}
Putting $u:=g_1(x)$ and $v:=g_2(y)$, and applying the definitions in \eq{phi}, we obtain equation \eq{aux}.
\end{proof}

In 2002, Daróczy and Páles studied equation \eq{aux} under the assumption that $\psi_1$ and $\psi_2$ are equal to a function $\psi$, which is a sum of a continuous function and a derivative. Moreover, they assumed that $\varphi$ can be expressed as $\psi+\alpha$ for some constant $\alpha\in\RR$. 

To address the equality problem of quasi-arithmetic and Cauchy means, the first author and Páles in 2019 considered a special case of \eq{aux}, again assuming $\psi_1=\psi_2$ and a closed zero set for $\varphi$. We note that the condition concerning the zero set is satisfied if, for example, $\varphi$ is continuous, or strictly monotone, or merely injective. 

Keeping the condition on the zero set of $\varphi$, the solutions to the general equation \eq{aux} were described in \cite{Kis24} by the first author in 2024.  

Finally, the second author \cite{Tot} described the solutions for measurable functions $\varphi$ and obtaining a result analogous to \cite{Kis24} in the case where $\varphi$ is a derivative. The following theorem originally appeared in the paper \cite{Tot} of the second author. For the sake of simpler phrasing, in the following theorem and in the two subsequent corollaries let $I_1,I_2\subseteq\RR$ be nonempty open intervals and $S:=\frac12(I_1+I_2)$.

\begin{thm}\label{peter}
Let $\left( \varphi, \psi_1, \psi_2\right)$ be a solution of the functional equation
\Eq{pet}{
\varphi\Big(\frac{x+y}2\Big)\big(\psi_1(x)-\psi_2(y)\big)=0,\qquad (x,y)\in I_1\times I_2
}
such that $\varphi$ has a primitive function on $S$. Then exactly one of the following statements is true. 
\begin{enumerate}\itemsep=1mm
\item\label{P1} The function $\varphi$ is identically zero on $S$ or there exists a constant $D\in\RR$ such that $\psi_1(x)=\psi_2(y)=D$ for all $x\in I_1$ and $y \in I_2$.
\item\label{P2} There exist constants $D,E\in\RR$ and, for $i=1,2$, there exist extended real numbers $a_i, b_i\in\RR\cup\{-\infty,+\infty\}$ such that 
$ \inf I_i \leq a_i \leq b_i \leq \sup I_i $ and for the open intervals 
\[
U_i := \left]\,\inf I_i, a_i \,\right[ 
\qquad \mbox{ and } \qquad 
V_i := \left] \, b_i, \sup I_i \, \right[ 
\]
the following assertions are fulfilled: 
\begin{itemize}\itemsep=1mm
\item $ U_i \neq I_i $ and $ V_i \neq I_i $ for any index $ i = 1,2 \,$; 
\item $ U_1 \,, U_2 $ are both nonempty or 
$ V_1 \,, V_2 $ are both nonempty; 
\item  $ \psi_1 (x) =  \psi_2 (y) = D $ for all $ (x,y) \in U_1 \times U_2 $ and 
       $ \psi_1 (x) =  \psi_2 (y) = E $ for all $ (x,y) \in V_1 \times V_2 \,$; 
\item $\varphi$ is identically zero on $\frac{1}{2} \left( K_1 + I_2 \right) \cup 
			\frac{1}{2} \left( I_1 + K_2 \right) $ 
			where 
			$ K_i := I_i \setminus \left( U_i \cup V_i \right) $ 
			is a nonempty subinterval 
			which is closed in $ I_i $, $ i=1,2 $. 
		\end{itemize}
		
		\item\label{P3} There exist a constant $ D \in \RR $ 
		and an index 
		$ j \in \lbrace 1,2 \rbrace $ such that 
		$ \psi_j(z) = D $ for all $ z \in I_j \,$, 
		while there exists a finite or countable collection of 
		disjoint, nonempty open intervals 
		$ U_n  \subset I_i $ 
		($ i \neq j $ and $ n = 1, \dots, N $ 
		where $ N \in \NN $ or $ N = \infty $) 
		with the following properties: 
		\Eq{*}{
			\bigcup_{n=1}^N U_n \neq I_i\quad\text{ and }\quad\psi_i(z) = D\text{ for all } z\in \bigcup_{n=1}^N U_n,}
		moreover  
		\Eq{*}{
			\varphi(t) = 0\text{ for all } t \in \frac{1}{2} \Big(\Big( I_i \setminus \bigcup_{n=1}^N U_n\Big) + I_j\Big).  
		}
	\end{enumerate} 
Conversely, if for the functions $ \psi_1 : I_1 \to \RR $, $ \psi_2 : I_2 \to \RR $ and $ \varphi : S \to \RR $ one of the three cases above holds, then the triplet $ \left( \varphi \,, \psi_1 \,, \psi_2 \right) $ is a solution of \eq{aux}.
\end{thm}

\begin{coro}\label{koviA}
Let $S\subseteq\RR$ be bounded from above and $(\varphi,\psi_1,\psi_2)$ be a solution of \eq{aux} such that $\varphi$ has a primitive function on $S$. Then at least one of the following assertions is true.
\begin{enumerate}[(i)]\itemsep=1mm
	\item\label{kovi1} There exists $p\in S$ such that $\varphi(x)=0$ for all $p<x<\sup S$.
	\item\label{kovi2} There exist $q_1\in I_1$, $q_2\in I_2$ and $E\in\RR$ such that $\psi_k(x)=E$ for all $q_k<x<\sup I_k$.
\end{enumerate}
\end{coro}

\begin{proof}
Under our conditions, for the functions $\varphi$, $\psi_1$ and $\psi_2$, we have the possibilities (\ref{P1}), (\ref{P2}) or (\ref{P3}) of Theorem \ref{peter}. 

Possibility (\ref{P1}) trivially implies the validity of \textit{(\ref{kovi1})} or \textit{(\ref{kovi2})}, so in that case we are done.

Assume that (\ref{P2}) of Theorem \ref{peter} holds. If $V_1$ and $V_2$ are nonempty, then \emph{(\ref{kovi2})} holds with $q_1:=b_1$ and $q_2:=b_2$. If $V_{i_0}$ is empty for some $i_0\in\{1,2\}$, then, particularly, $U_{i_0}$ is non-empty. In that case \emph{(\ref{kovi1})} holds with $p:=\inf\frac12\big((I_{i_0}\setminus U_{i_0})+I_{3-{i_0}}\big)$.

Suppose finally that (\ref{P3}) of Theorem \ref{peter} holds. For simplicity, assume that $j=1$. If $(\ref{kovi1})$ is true, then we are done, hence we may also assume that $(\ref{kovi1})$ does not hold. Observe that, by our assumption on $S$, we must have $\sup I_1,\sup I_2\in\RR$. Let $x_0\in I_1$ be any element with $\sup I_1-x_0<\diam I_2$ and let $u_0:=\frac12(x_0+\sup I_2)$. Then we must have an element $u_0<v_0<\sup S$ such that $\varphi(v_0)\neq 0$. Now consider the interval $H:=(2v_0-I_2)\cap I_1\subseteq I_1$.

It is obvious that $\psi_1$ is identically $D$ on $H$. Indeed, let $x\in H$ be any. Then $x\in I_1$ and there exists $y\in I_2$, such that $v_0=\frac12(x+y)$ holds. Substituting $x$ and $y$ into our equation \eq{pet}, we obtain that $\varphi(v_0)(\psi_1(x)-\psi_2(y))=0$. This implies that $\psi_1(x)=\psi_2(y)=D$.

Then, on the one hand, it is easy to see that $\diam(2v_0-I_2)=\diam I_2$. On the other hand, $\inf(2v_0-I_2)>2u_0-\sup I_2=x_0$. Consequently, $\sup H=\sup I_1$. Therefore \emph{(\ref{kovi2})} holds true with $q_1:=\inf H$ and with any fixed $q_2\in I_2$.
\end{proof}

\begin{coro}\label{koviB}
Let $S\subseteq\RR$ be bounded from below and $(\varphi,\psi_1,\psi_2)$ be a solution of \eq{aux} such that $\varphi$ has a primitive function on $S$. Then at least one of the following assertions is true.
\begin{enumerate}[(i)]\itemsep=1mm
	\item There exists $p\in S$ such that $\varphi(x)=0$ for all $\inf S<x<p$.
	\item There exist $q_1\in I_1$, $q_2\in I_2$ and $D\in\RR$ such that $\psi_k(x)=D$ for all $\inf I_k<x<q_k$.
\end{enumerate}
\end{coro}

\begin{proof}
Naturally, the proof parallels that of Corollary \ref{koviA}.
\end{proof}

\subsection{Complementarity Results}\label{PA2}

In Section \ref{A} (cf. Remark \ref{freedom}), we already mentioned that the behavior of the function $F$, in a certain sense, is closely related to that of the functions $g_1$ and $g_2$. Roughly speaking, as long as one of them is affine, the other can be chosen freely. In the following, we examine this relationship over subintervals. It turns out that as soon as $F$ ceases to be affine, the functions $g_1$ and $g_2$ become affine, and vice versa. We begin by focusing on the function $F$.

\begin{prop}\label{gaf}
Let $(F,f_k,G,g_k)$ be a solution of the equation \eqref{eq-Invariance-eq}, such that each of its coordinate functions is differentiable and $g_1$ and $g_2$ are strictly monotone in the same sense. Let further $U\in\mathscr{A}_F^{\max}$ with $U\neq I$.
\begin{enumerate}\itemsep=1mm
\item\label{1gaf} If $\inf U\in I$, then there exist constants $D,\delta_1,\delta_2\in\RR$ with $D\neq 0$ such that $\mathscr{A}_{g_1}(D,\delta_1)\cap\mathscr{A}_{g_2}(D,\delta_2)$ contains an open neighborhood of $\inf U$. 
\item\label{2gaf} If $\sup U\in I$, then there exist constants $E,\epsilon_1,\epsilon_2\in\RR$ with $E\neq 0$ such that $\mathscr{A}_{g_1}(E,\epsilon_1)\cap\mathscr{A}_{g_2}(E,\epsilon_2)$ contains an open neighborhood of $\sup U$. 
	\end{enumerate}
\end{prop}

\begin{proof}
	Let $F(x)=Ax+b$ if $x\in U$ for some constants $A,b\in\RR$. We note that $F(x)=Ax+b$ holds for $x\in\mathrm{cl}_I(U)$ as well. Assume that $p:=\inf U\in I$ holds. The case $\sup U \in I$ can be handled in a similar manner of reasoning.
	
	\emph{In the first step, we construct a left-sided open neighborhood $L$ of $p$ on which $g_1$ and $g_2$ are affine with the same non-zero slope.} On the one hand, in view of Proposition \ref{Aff}, there exist constants $B,\beta\in\RR$ such that $G(u)=Bu+\beta$ for all $u\in g_1(U)+g_2(U)$. On the other hand, by Proposition \ref{aux}, we get that
	\Eq{auxU}{
		\varphi\Big(\frac{u+v}2\Big)\big(\psi_1(u)-\psi_2(v)\big)=0,\qquad
		(u,v)\in g_1(U^*)\times g_2(U^*),
	}
	holds with
	\Eq{*}{
		\varphi(t)=G'(2t)-B\quad\text{if } t\in\tfrac12\big(g_1(U^*)+g_2(U^*)\big)\quad\text{and}\quad
		\psi_k(s)=g_k'\circ g_k^{-1}(s)\quad\text{if } s\in g_k(U^*),
	}
	where $U^*:=U_\mathrm{ext}\cap U_\mathrm{ref}$. Due to our assumption $U\neq I$, the inclusion $U\subseteq U^*$ is strict. From this point on, we assume that $g_1$ and $g_2$ are strictly increasing. The strictly decreasing case can be handled similarly, with obvious modifications.

We claim that there is no open subinterval $W\subseteq U^*$ of positive length with $\inf I<\inf W=:\xi$ such that $\sup W=p$ and $\varphi(t)=0$ whenever $t\in\frac12\big(g_1(W)+g_2(W)\big)$. If this were not the case, then we would have $G'(t)=B$ for all $t\in g_1(W)+g_2(W)$. Consequently, there would exist $\beta^*\in\RR$ such that $G(t)=Bt+\beta^*$ if $t\in g_1(W)+g_2(W)$. Utilizing the facts that $W$ is bounded from above, and $U$ is bounded from below, we obtain that
\Eq{*}{
\sup\big(g_1(W)+g_2(W)\big)=g_1(p)+g_2(p)=\inf\big(g_1(U)+g_2(U)\big).}
Since $G$ is continuous at the point $g_1(p)+g_2(p)$, it follows that $\beta^*=\beta$. Thus we obtained that $G(u)=Bu+\beta$ holds whenever $u\in g_1(J)+g_2(J)$, where $J:=\conv(U\cup\{\xi\})$. Combining this with the fact that \eq{fk+} of Proposition \ref{ext} holds over $J$, by Proposition \ref{Gmax}, we obtain that $\xi=\inf U$, contradicting the fact that $W$ was a nonempty open interval.

In the next step, let $I_1:=g_1(U^-\cap U^*)$ and $I_2:=g_2(U^-\cap U^*)$. In light of the above, if we apply Corollary \ref{koviA} for $S:=\frac12(I_1+I_2)$, then we obtain that case \emph{(\ref{kovi2})} of Corollary \ref{koviA} must be valid. Thus there exist $q_1\in I_1$, $q_2\in I_2$ and $D\in\RR$ such that
\Eq{*}{
D=\psi_k(u)=g_k'\circ g_k^{-1}(u),\qquad q_k<u<\sup I_k=g_k(p).
}
Note that the strict monotonicity of $g_1$ implies that $D\neq 0$. Thus $g_k'(x)=D$ whenever $g_k^{-1}(q_k)<x<p$, and hence there exist constants $\delta_1,\delta_2\in\RR$ such that $g_k(x)=D x+\delta_k$ if $g_k^{-1}(q_k)<x<p$. Moreover, because of the continuity of $g_1$ and $g_2$, we also have $g_k(p)=D p+\delta_k$ for $k=1,2$. Thus let
\Eq{*}{L:=\,]\,g_1^{-1}(q_1),p\,]\,\cap\,]\,g_2^{-1}(q_2),p\,]\subseteq U^*.}
	
\emph{In the second step we show that the above left-sided open neighborhood $L$ of $p$ can be extended to the right side of the point $p$ as well. This continuation of $L$ will be denoted by $R$.} 

Let $r:=2p-\max\big(g_1^{-1}(q_1), g_2^{-1}(q_2)\big)=2p-\inf L$. We show that $g_1'$ and $g_2'$ equal identically $D$ on the interval $R:=\,]\,p,r\,[\,\cap\,U^*$. Differentiating equation \eqref{eq-Invariance-eq} with respect to its second variable, for $x\in R$ and $y\in L$, we get that
\Eq{*}{
\frac12 F'\Big(\frac{x+y}2\Big)+f_2'(y)=D G'\big(g_1(x)+g_2(y)\big).
}
Therefore, if we differentiate \eqref{eq-Invariance-eq} with respect to its first variable, by $D\neq 0$, it follows that
\Eq{*}{
\big(D-g_1'(x)\big)F'\Big(\frac{x+y}2\Big)=2g_1'(x)f_2'(y)-2D f_1'(x)
}
is valid for $x\in R\subseteq U^*$ and $y\in L\subseteq U^*$. Then, using formula \eq{fk+} of Proposition \ref{ext} over $U^*$, we can simplify the above equation to
\Eq{*}{
\big(D-g_1'(x)\big)\Big(F'\Big(\frac{x+y}2\Big)-A\Big)=0,\qquad (x,y)\in R\times L.
}
	
Consequently, if there was a point $x_0$ in $R$ for which $g_1'(x_0)\neq D$ then $F$ would be affine on the interval $\frac12(x_0+L)$. Its supremum $M$ is nothing else but $\frac12(x_0+p)$, hence $p<M$. For its infimum $m$, the inequality $m<p$ is equivalent to $x_0<r$, which is true. The inequalities $m<p<M$ contradict the maximality of $U$. \emph{Summarizing the above, we have thus obtained that $L\cup R\in\mathscr{A}_{g_1}(D,\delta_1)$.}

The statement concerning $g_2$ follows analogously by exchanging the role of the variables $x$ and $y$ in the above reasoning.
\end{proof}

We now formulate the counterpart of the above proposition as well.

\begin{prop}\label{gaf2}Let $(F,f_k,G,g_k)$ be a solution of the equation \eqref{eq-Invariance-eq}, such that each of its coordinate functions is differentiable. Let further $U_k\in\mathscr{A}^{\max}_{g_k}(D,\delta_k)$ for some constants $D,\delta_1,\delta_2\in\RR$ with $D\neq 0$. Finally assume that $U:=U_1\cap U_2\neq I$.
\begin{enumerate}
\item\label{1gaf2} If $\inf U\in I$, then $\mathscr{A}_F$ contains an open neighborhood $V$ of $\inf U$ for which
\Eq{*}{
\tfrac12(\inf U+\sup U)\leq \sup V
}
holds true. 
\item\label{2gaf2} If $\sup U\in I$, then $\mathscr{A}_F$ contains an open neighborhood $V$ of $\sup U$ for which
\Eq{*}{
\inf V\leq \tfrac12(\inf U+\sup U)
}
holds true. 
\end{enumerate}
Consequently, if $\inf I<\inf U$ and $\sup U<\sup I$, then $F$ is affine on the subinterval $U$.
\end{prop}

\begin{proof}
The proof will be carried out only for the case $p:=\inf U\in I$. The remaining case $\sup U\in I$ is to be treated similarly. 

By $\inf U=\max(\inf U_1,\inf U_2)$, there exist an index $k_0\in\{1,2\}$ and a sequence $(x_n):\NN\to I\setminus\{p\}$ be such that $x_n\to p$ from left as $n\to\infty$ and that $g_{k_0}'(x_n)\neq D$ for all $n\in\NN$. In view of Proposition \ref{gkc}, there exists $C\in\RR$ such that $f_{3-k_0}'(y)=C$ whenever $y\in U$. Thus, differentiating equation \eqref{eq-Invariance-eq} independently with respect to each of its variables, for fixed $n\in\NN$ and for any $y\in U$, we obtain the system
\begin{align*}
\frac12F'\Big(\frac{x_n+y}2\Big)+f_{k_0}'(x_n)&=G'\big(g_{k_0}(x_n)+g_{3-k_0}(y)\big)g_{k_0}'(x_n)\\
\frac12F'\Big(\frac{x_n+y}2\Big)+C&=D G'\big(g_{k_0}(x_n)+g_{3-k_0}(y)\big).
\end{align*}
Then, using that $D\neq0$, an obvious computation yields that
\Eq{*}{
F'\Big(\frac{x_n+y}2\Big)=\frac{2Cg_{k_0}'(x_n)-2D f_{k_0}'(x_n)}{D-g_{k_0}'(x_n)},\qquad y\in U.
}
Consequently, for any fixed $n\in\NN$, the function $F$ is affine on the open subinterval $U_n:=\frac12(x_n+U)\subseteq I$.
	
On the one hand, $x_n<p=\inf U$ for all $n\in\NN$, hence $\inf U_n<p$ for all $n\in\NN$. On the other hand, it is easy to see that $\diam(U_n)=\frac12\diam(U)$ for all $n\in\NN$. Hence we must have an index $n_0\in\NN$ such that $p\in U_k$ if $k\geq n_0$. Without loss of generality, we can assume that $p\in U_n$ for all $n\in\NN$. 

Since the family $\{\, \left] \, p,\sup U_n \, \right[ \, :  n\in\NN\}$ is a chain of open intervals, there exist constants $A,b\in\RR$ such that $F(x)=Ax+b$ whenever $x\in\,\left] \, p,\sup U_n \, \right[ \,$ and $n\in\NN$. Consequently, $F(x)=Ax+b$ holds if
\Eq{*}{
x\in\bigcup\limits_{n\in\NN}\,\left]\,  p,\sup U_n \, \right[\,=
\, \left] \, \inf U,\tfrac12(\inf U+\sup U)\, \right[ \,.
}
Then there uniquely exists $V\in\mathscr{A}^{\max}_F(A,b)$ with $\left] \, \inf U,\tfrac12(\inf U+\sup U) \, \right[\,\subseteq V$. By Proposition \ref{gaf}, it also follows that $\inf V<\inf U$ must hold. Consequently, $V$ is an open neighborhood of $\inf U$.

Finally, if $\inf I<\inf U$ and $\sup U<\sup I$, then $\inf U,\sup U\in\RR$. Let $u:=\tfrac12(\inf U+\sup U)$. In view of statements (1) and (2) of our current proposition, there exist constants $A_1,A_2,b_1,b_2\in\RR$ such that, for $x\in U$, we have
\Eq{*}{
F(x)=A_1x+b_1\quad \text{ if } x<u
\qquad\text{and}\qquad 
F(x)=A_2x+b_2\quad \text{ if }x>u.}
By the differentiability of $F$ at $u$, we get that $A_1=F'(u)=A_2$, and hence $b_1=b_2$ follows as well.
\end{proof}

\begin{coro}\label{glue1}
Let $(F,f_k,G,g_k)$ be a solution of equation \eqref{eq-Invariance-eq} such that each of its coordinate functions is differentiable, $F$ is partially affine, and $g_1$ and $g_2$ are strictly monotone in the same sense. Then there exist constants $D,\delta_1,\delta_2\in\RR$ with $D\neq 0$ and $U\in\mathscr{A}_{g_1}(D,\delta_1)\cap\mathscr{A}_{g_2}(D,\delta_2)$ such that
\Eq{*}{
\inf U=\inf I\qquad\text{or}\qquad\sup U=\sup I.
}
\end{coro}

\begin{proof}
Let $J\in\mathscr{A}^{\max}_F$ be any. By our assumption $J\neq I$, thus we may assume that $p:=\inf J\in I$. Then, by Proposition \ref{gaf}, there exist constants $D,\delta_1,\delta_2\in\RR$ with $D\neq 0$ such that $\mathscr{A}_{g_1}(D,\delta_1)\cap\mathscr{A}_{g_2}(D,\delta_2)$ contains an open neighborhood $P$ of $p$. For $k=1,2$, let $U_k\in\mathscr{A}^{\max}_{g_k}(D,\delta_k)$ be the unique interval for which $P\subseteq U_k$ holds. Let finally $U:=U_1\cap U_2$. Then $U\in\mathscr{A}_{g_1}(D,\delta_1)\cap\mathscr{A}_{g_2}(D,\delta_2)$.

If $U=I$, then our statement follows easily. If $U\neq I$, then, by Proposition \ref{gaf2}, we must have either $\inf U=\inf I$ or $\sup U=\sup I$, otherwise $J$ would be extendable.
\end{proof}

\subsection{Main Result of the Partially Affine Case}\label{PA3} Before we formulate the main result of Section \ref{PA}, we prove a lemma concerning the sets defined in \eq{H-+}.
\begin{lemm}\label{H}If $H\subseteq I$ is a nonempty subinterval, then
\Eq{HIncl}{
H^-+H^+\subseteq H+I.}
\end{lemm}

\begin{proof} If at least one of the sets $H^-$ or $H^+$ is empty, then $H^-+H^+$ is empty, hence in that case the above inclusion trivially holds. Assume that none of the sets $H^-$ and $H^+$ is empty. Then $H\neq I$ and we have
\Eq{*}{
\inf(H+I)=\inf H+\inf H^-\leq\sup H+\inf H^-=\inf(H^-+H^+)
}
and
\Eq{*}{
\sup(H^-+H^+)=\inf H+\sup H^+\leq\sup H+\sup H^+=\sup(H+I).
}
Having intervals on both sides of \eq{HIncl}, these yield the validity of our statement.
\end{proof}

To simplify the formulation of the proof of our main theorem, we introduce the following notation. Assume that $g_1$ and $g_2$ are differentiable on $I$. For a given constant $D\in\RR$ with $D\neq 0$, let $\mathscr{L}(D)$ stand for the intersection of the intervals $U_1\in\mathscr{A}^{\max}_{g_1}$ and $U_2\in\mathscr{A}^{\max}_{g_2}$ for which $\inf U_1=\inf U_2=\inf I$ and $g_1'(x)=g_2'(y)=D$ hold for all $x\in U_1$ and $y\in U_2$, provided that such intervals exist. Otherwise $\mathscr{L}(D)$ is empty.

In a similar manner, $\mathscr{R}(D)$ stands for the intersection of the intervals $U_1\in\mathscr{A}^{\max}_{g_1}$ and $U_2\in\mathscr{A}^{\max}_{g_2}$ for which $\sup U_1=\sup U_2=\sup I$ and $g_1'(x)=g_2'(y)=D$ hold for all $x\in U_1$ and $y\in U_2$, provided that such intervals exist. Otherwise $\mathscr{R}(D)$ is empty.

Now we can formulate the main result of the current section.

\begin{remark}
To avoid misunderstandings regarding the notation used in the next theorem, we would like to clarify the following. If a $+$ or $-$ symbol appears in the superscript of a subinterval $K$ of $I$, then definition \eq{H-+} must be applied. If a real number carries a $+$ or $-$ in its superscript, this indicates that the corresponding coefficient (more precisely, the associated representation of the function in question) is active or valid over the sets $K^+$ or $K^-$, respectively. This will be, incidentally, also evident from the formulation of the theorem.

Finally, $\overline{K}$ will stand for the arithmetic mean $\frac12(K+I)$.
\end{remark}

\begin{thm}\label{M3}
Let $(F,f_k,G,g_k)$ be such that each of its coordinate functions is differentiable, $F$ is partially affine on $I$ and that $g_1,g_2:I\to\RR$ are strictly monotone in the same sense. Then $(F,f_k,G,g_k)$ solves equation \eqref{eq-Invariance-eq} if and only if there exist an interval $K\subset I$ of positive length, which is closed in $I$ and constants $A,B,C^-,C^+,D^-,D^+\in\RR$ and $\alpha,\beta_k,\gamma_k^-,\gamma_k^+,\delta_k^-,\delta_k^+\in\RR$ with
\Eq{Const}{
D^-D^+\neq 0,
\qquad C^\pm+\frac12A=BD^\pm
\qquad\text{and}\qquad
\gamma_k^\pm+\frac12\alpha=B\delta_k^\pm+\beta_k
}
such that
\Eq{-}{
f_k(s)=C^-s+\gamma_k^-,\quad 
G(u)=F\Big(\frac{u-\delta^-}{2D^-}\Big)+\frac{C^-}{D^-}(u-\delta^-)+\gamma^-,\quad\text{and}\quad
g_k(s)=D^-s+\delta_k^-
}
if $s\in K^-$ and $u\in g_1(K^-)+g_2(K^-)$,
\Eq{k}{
F(x)=Ax+\alpha,\qquad f_k(x)=-\tfrac12 F(x)+Bg_k(x)+\beta_k\qquad\text{and}\qquad G(u)=Bu+\beta
}
if $x\in\overline{K}$ and $u\in g_1(\overline{K})+g_2(\overline{K})$, and
\Eq{+}{
f_k(t)=C^+t+\gamma_k^+,\quad 
G(w)=F\Big(\frac{w-\delta^+}{2D^+}\Big)+\frac{C^+}{D^+}(w-\delta^+)+\gamma^+,\quad\text{and}\quad
g_k(t)=D^+t+\delta_k^+
}
if $t\in K^+$ and $w\in g_1(K^+)+g_2(K^+)$, where 
\Eq{*}{
\beta=\beta_1+\beta_2,\qquad \gamma^\pm=\gamma_1^\pm+\gamma_2^\pm\qquad\text{and}\qquad \delta^\pm=\delta_1^\pm+\delta_2^\pm,}
furthermore
\Eq{G}{
G(v)=Bv+\beta,\qquad v\in g_1(H)+g_2(J),
}
where $H,J\in\{K^-,K,K^+\}$ with $H\neq J$.
\end{thm}

\begin{proof}
\emph{First we deal with necessity.} By Corollary \ref{glue1}, it follows that there exist constants $D^-,D^+\in\RR$ with $D^-D^+\neq 0$ such that at least one of the open intervals $\mathscr{L}(D^-)$ and $\mathscr{R}(D^+)$ is not empty. Consequently,
\Eq{*}{
K:=I\setminus\big(\mathscr{L}(D^-)\cup\mathscr{R}(D^+)\big)
}
is a closed subinterval of $I$ which may be empty, but is certainly not equal to $I$. It is easy to observe that $K^-=\mathscr{L}(D^-)$ and $K^+=\mathscr{R}(D^+)$ hold.

Now we prove that $K$ cannot be a singleton, that is, of the form $\{p\}$. If it was, then $p=\sup K^-=\inf K^+$ would follow. Thus both of $K^-$ and $K^+$ are nonempty and hence there exist constants $\delta_1^-,\delta_2^-,\delta_1^+,\delta_2^+\in\RR$ such that $g_1(x)=D^-x+\delta_1^-$ and $g_2(x)=D^-x+\delta_2^-$ if $x\in K^-$ and $g_1(y)=D^+y+\delta_1^+$ and $g_2(y)=D^+y+\delta_2^+$ if $y\in K^+$. Due to the continuity of the functions in question, these formulas remain valid at the point $p$, from which, by the differentiability property, it follows that $D^-=D^+=:D$. Hence $\delta_1^-=\delta_1^+$ and $\delta_2^-=\delta_2^+$ also hold true. Consequently we have $g_1(x)=Dx+\delta_1$ and $g_2(x)=Dx+\delta_2$ whenever $x\in I$, contradicting that $\inf I<\inf K^+$ and $\sup K^-<\sup I$. In light of this, we can distinguish the following main cases.

\emph{Case 1.} If $K$ is empty then $K^-=K^+=I$ and $D^-=D^+=:D$ follow. Moreover $\overline{K}$ and $g_1(\overline{K})+g_2(\overline{K})$ are also empty. Furthermore there exist $\delta_1,\delta_2\in\RR$ such that $g_1(x)=Dx+\delta_1$ and $g_2(x)=Dx+\delta_2$ if $x\in I$. Then, applying Proposition \ref{gkc} under $U=I$, we get that there exist $C\in\RR$ and $\gamma_1,\gamma_2\in\RR$ such that \eq{-} and \eq{+} hold with $C^-=C^+=:C$, $\gamma_1^-=\gamma_1^+=:\gamma_1$, $\gamma_2^-=\gamma_2^+=:\gamma_2$, $\gamma^-=\gamma^+=\gamma_1+\gamma_2$, and $\delta^-=\delta^+=\delta_1+\delta_2$.
	
\emph{Case 2.} Assume that $K$ is a closed subinterval of positive length. Since $K\neq I$, we may also assume that $p:=\inf K=\sup K^-$ belongs to $I$. The case $\sup K \in I$ can be handled in a completely analogous manner. Then $K^-$ is not empty. Applying Proposition \ref{gkc} on $K^-$, we obtain that there exist constants $C^-\in\RR$ and $\gamma_1^-,\gamma_2^-,\delta_1^-,\delta_2^-\in\RR$ such that \eq{-} holds with $\gamma^-=\gamma_1^-+\gamma_2^-$ and $\delta^-=\delta_1^-+\delta_2^-$.
	
Due to $K^-=\mathscr{L}(D^-)$ and applying Proposition \ref{gaf2}, there exists an open neighborhood $U\subseteq I$ of $p$ such that $F$ is affine on $U$. We may assume that $U\in\mathscr{A}^{\max}_F$. Note that, in view of statement (2) Proposition \ref{gaf2}, we have
\Eq{ineq}{
\inf U\leq\frac12(\inf K^-+\sup K^-)=\frac12(\inf I+p)=\inf\overline{K}.}

Now we assert that $K\subseteq U$. This is trivial if $\sup U=\sup I$ holds. Thus we may and do assume that $\sup U<\sup I$. Then, to prove inclusion $K\subseteq U$, it is enough to show that $\sup K<\sup U$. Assume that this is not the case, that is, we have $\sup U\leq\sup K$. The interval $U$ was maximal and, by our assumption $\sup U<\sup I$, condition $U\neq I$ follows. Thus, by Proposition \ref{gaf}, we get that there exist open neighborhoods $V_1\in\mathscr{A}^{\max}_{g_1}$ and $V_2\in\mathscr{A}^{\max}_{g_2}$ of $\sup U$ such that $g_1$ and $g_2$ share the same nonzero slope on the intersection $V:=V_1\cap V_2$. By $K^-=\mathscr{L}(D^-)$, we must have $\sup K^-<\inf V$, which implies that $\inf I<\inf V$. If we had $\sup V\in I$ then, by Proposition \ref{gaf2}, we would have $U\cup V\in\mathscr{A}_F$ contradicting $U\in\mathscr{A}^{\max}_F$. Consequently we have $\sup V=\sup I$ and hence $V\subseteq K^+$. By our assumption $\sup U\leq \sup K$, it follows that $V\cap K$ is not empty. This leads to the contradiction
\Eq{*}{
\emptyset\neq V\cap K\subseteq K^+\cap K=\emptyset,
}
which shows that $K\subseteq U$ is indeed valid.

Now we are going to prove that $\overline{K}\subseteq U$ holds too. Depending on whether $K^+$ is empty or not, we can distinguish two further subcases.

\emph{Case 2.1.} If $K^+$ is empty, then $\sup K=\sup I$. Consequently we have $\sup\overline{K}=\sup I$, furthermore, by $K\subseteq U$, we have $\sup U=\sup I$. Combining this with inequality \eq{ineq}, the inclusion $\overline{K}\subseteq U$ follows.
	
\emph{Case 2.2.} If $K^+$ is not empty, then by Proposition \ref{gkc} we get that there exist $C^+\in\RR$ and 
$\gamma_1^+,\gamma_2^+,\delta_1^+,\delta_2^+\in\RR$ such that \eq{+} holds with $\gamma^+=\gamma_1^++\gamma_2^+$ and 
$\delta^+=\delta_1^++\delta_2^+$. If $\sup U=\sup I$, then our statment is trivial. Therefore assume that $\sup U<\sup I$. Hence 
there must exist an open neighborhood $V$ of $\sup U$ on which $g_1$ and $g_2$ are affine functions with a common non-zero slope. Since
\Eq{*}{
\sup K^-=p=\inf K<\sup K\leq \sup U,
}
it follows that $\sup V=\sup I$, hence $\sup U\in K^+$. By statement (1) of Proposition \ref{gaf2} and inequality \eq{ineq}, we get that
\Eq{*}{
\inf U\leq\inf\overline{K}<\sup\overline{K}=\frac12(\sup K+\sup I)=\frac12(\inf K^++\sup K^+)\leq\sup U.
}
Thus $\overline{K}\subseteq U$ and hence, by Proposition \ref{Aff}, there exist constants $A,B\in\RR$ and $\alpha,\beta_1,\beta_2\in\RR$ such that \eq{k} holds with $\beta=\beta_1+\beta_2$.

\emph{Now we can turn to the verification of the restrictions on some constants listed in \eq{Const}.} The validity of condition $D^-D^+\neq 0$ was concluded in the very first step. If either $K^-$ or $K^+$ is empty then the corresponding constant can be considered as any nonzero number.
	
If $K$ is empty, then let $A,\alpha\in\RR$ be arbitrarily fixed and then define $B:=\frac1D(C+\frac12A)$ and $\beta_k:=\gamma_k-\frac12\alpha-\frac{\delta_k}D(C+\frac12A)$, where $C, D, \gamma_k$ and $\delta_k$ were introduced in \emph{Case 1}. Then the second and third identities in \eq{Const} also are valid. If the intersection $K^-\cap\overline{K}$ is not empty, then the formulae in \eq{-} and \eq{k} are active and imply that
\Eq{*}{
\Big(C^-+\frac12 A-BD^-\Big)x+\Big(\gamma_k^-+\frac12\alpha-B\delta_k^--\beta_k\Big)=0,\qquad x\in 
K^-\cap\overline{K}.
}
This yields that the first half of the second and third identities in \eq{Const} are valid. The second half can be obtained using the nonempty interval $K^+\cap\overline{K}$ in the same way provided that it is not empty.

\emph{Finally we focus on assertion \eq{G}.} By symmetry considerations, it is obviously sufficient to deal with the cases 
$(H,J)\in\{(K^-,K),(K^+,K),(K^-,K^+)\}$. Furthermore, the proofs for $(K^-,K)$ and $(K^+,K)$ are completely analogous. Therefore we perform the proof only for the cases when $(H,J)=(K^-,K)$ or $(H,J)=(K^-,K^+)$.
	
Let $(H,J):=(K^-,K)$. If one of $K^-$ or $K$ is empty, then $g_1(K^-)+g_2(K)$ is empty and there is nothing to prove. Therefore we may assume that none of them is empty. Let $v\in g_1(K^-)+g_2(K)$ be any point. Then there exist $x\in K^-$ and $y\in K$ such that $v=g_1(x)+g_2(y)$. Then, considering that $\frac12(K^-+K)\subseteq\frac12(I+K)=\overline{K}$, equation \eqref{eq-Invariance-eq} reduces to
\Eq{*}{
\Big(C^-+\frac12 A\Big)x+\gamma_1^-+\frac12\alpha+Bg_2(y)+\beta_2=G(v).
}
Applying $\frac12A+C^-=BD^-$ and $\gamma_1^-+\frac12\alpha=B\delta_1^-+\beta_1$, we obtain that
\Eq{*}{
G(v)=BD^-x+B\delta_1^-+Bg_2(y)+\beta=B\big(D^-x+\delta_1^-+g_2(y)\big)+\beta=Bv+\beta.
}
	
Let $(H,J):=(K^-,K^+)$. Again, we may assume that none of them is empty. Let $v\in g_1(K^-)+g_2(K)$ be any. Then there are $x\in K^-$ and $y\in K^+$ such that $v=g_1(x)+g_2(y)$. By Lemma \ref{H}, it follows that $\frac12(K^-+K^+)\subseteq \overline{K}$. Therefore equation \eqref{eq-Invariance-eq} can be written as
\Eq{*}{
\frac{A}2(x+y)+\alpha+C^-x+\gamma_1^-+C^+y+\gamma_2^+=G(v).
}
Applying the identities in \eq{Const}, we obtain that
\Eq{*}{
G(v)=BD^-x+B\delta_1^-+BD^+y+B\delta_2^++\beta=Bv+\beta.
}

\emph{Now we can turn to the sufficiency.} Let $x,y\in I$ be any points. Without loss of generality, we may assume that $x\leq y$. If $x,y\in K^-$ or $x,y\in K^+$, then, by Proposition \ref{gkc}, the functions listed in \eq{-} or \eq{+} solve equation \eqref{eq-Invariance-eq} with arbitrary function $F$. Hence we may assume that $x\leq \sup K$ and $\inf K\leq y$. If $x,y\in K$, then by Proposition \ref{Aff} the functions listed in \eq{k} solve equation \eqref{eq-Invariance-eq} with arbitrary functions $g_1$ and $g_2$. Therefore assume that $\sup K<y$ or $x<\inf K$ holds. We are going to focus on the case when $x<\inf K$, that is, when $x\in K^-$. The proof for the case $\sup K<y$ is completely analogous.
	
Assume that $y\in K$. Then $y\in\overline{K}$ and $\frac12(x+y)\in\frac12(K^-+K)\subseteq\overline{K}$ hold. Starting from the left hand side of equation \eqref{eq-Invariance-eq}, we get that
	\Eq{*}{
		F\Big(\frac{x+y}2\Big)+f_1(x)+f_2(y)
		&=\Big(C^-+\frac12 A\Big)x+\gamma_1^-+\frac12\alpha+Bg_2(y)+\beta_2\\
		&=BD^-x+B\delta_1^-+Bg_2(y)+\beta=B\big(g_1(x)+g_2(y)\big)+\beta\\
		&=G\big(g_1(x)+g_2(y)\big),
	}
	where in the last step we have taken into account that $g_1(x)+g_2(y)\in g_1(K^-)+g_2(K)$. Hence in this case the equality is satisfied.
	
	Assume that $y\in K^+$. Then $\frac12(x+y)\in\frac12(K^-+K^+)$, thus, by Lemma \ref{H}, we have $\frac12(x+y)\in\overline{K}$. Therefore, starting again from the right hand side of equation \eqref{eq-Invariance-eq}, we obtain
	\Eq{*}{
		F\Big(\frac{x+y}2\Big)+f_1(x)+f_2(y)
		&=\Big(C^-+\frac12A\Big)x+\Big(C^++\frac12A\Big)y+\alpha+\gamma_1^-+\gamma_2^+\\
		&=BD^-x+BD^+y+B\delta_1^-+B\delta_2^++\beta\\
		&=B\big(g_1(x)+g_2(y)\big)+\beta=G\big(g_1(x)+g_2(y)\big),
	}
	where in the last step we applied the fact that $g_1(x)+g_2(y)\in g_1(K^-)+g_2(K^+)$.
\end{proof}

\begin{exmp}
To end this section we explicitly formulate a particular 
partially affine solution. 
For this purpose, in Theorem \ref{M3} we have to specify the 
interval $ K $ and all the constants. In the following 
example let 
$ I = \left] \, 0,4 \, \right[ \,$, 
$ K = \left[ 2,4 \, \right[ \, $, thus 
$ K^- = \left] \, 0,2 \, \right[ \,$, 
$ K^+ = \emptyset $ and 
$ \overline{K} = \left] \, 1,4 \, \right[ \,$. 
Moreover, $ A = 4 $, $ B = 3 $, 
$ C^- = D^- = 1 $, 
$ \alpha = \beta_k = \beta = \gamma_k^- = \gamma^- 
= \delta_k^- = \delta^- = 0 $ while 
$ C^+, D^+, \gamma_k^+ , \delta_k^+ $ are arbitrary 
such that \eq{Const} holds. 
Thus 
\begin{align*}
F(x) = 
\begin{cases} 
2x^2 + 2 \hspace{2mm} & \mbox{if } 0 < x \leq 1, \\ 
4x \hspace{2mm} & \mbox{if } 1 < x < 4;  
\end{cases} 
\hspace{15mm}
f_k(x) = 
\begin{cases} 
x \hspace{2mm} & \mbox{if } 0 < x \leq 2, \\ 
\frac{3}{4}x^2 - 2x + 3 \hspace{2mm} & \mbox{if } 2 < x < 4;  
\end{cases} \\ 
g_k(x) = 
\begin{cases} 
x \hspace{2mm} & \mbox{if } 0 < x \leq 2, \\ 
\frac{1}{4}x^2 + 1 \hspace{2mm} & \mbox{if } 2 < x < 4; 
\end{cases} 
\hspace{15mm}
G(u) = 
\begin{cases} 
\frac{1}{2}u^2 + u + 2 \hspace{2mm} & \mbox{if } 0 < u \leq 2, \\ 
3u \hspace{2mm} & \mbox{if } 2 < u < 10.   
\end{cases}
\end{align*} 
One can easily check that these functions are exactly once 
continuously differentiable. 
\end{exmp}

\section{Solutions, where $F$ is nowhere affine}\label{NA}
%
%
%

In the sequel we would like to determine the solutions of 
equation \eqref{eq-Invariance-eq} 
in the only remaining case where 
we assume that $ F $ is not affine on any 
nonempty open interval $ L \subseteq I $. 
It is easy to see that 
$ g_1'(x) \neq 0 $ and $ g_2'(x) \neq 0 $ 
for every $ x \in I $. 

\begin{prop}
Let $ \left( F, f_k \,, G, g_k \right) $ 
be a solution of the equation 
\eqref{eq-Invariance-eq}, 
such that each of its coordinate functions is differentiable. 
Suppose that $ \mathscr{A}_F = \emptyset $. 
Then $ g_1'(x) \neq 0 $ and $ g_2'(x) \neq 0 $ 
for every $ x \in I $. 
\end{prop}

\begin{proof}
We only show that $ g_1'(x) \neq 0 $ for all $ x \in I $, 
the proof of the other claim is analogous. 
Differentiating \eqref{eq-Invariance-eq} with respect 
to $ x $, we obtain 
\[
\frac{1}{2} F' \Big( \frac{x+y}{2} \Big) 
+ f_1'(x) = G' \big(g_1(x)+g_2(y) \big) \cdot g_1'(x) 
\]
for all $ x,y \in I$. Now if $ g_1'(x_0) = 0 $ 
for some $ x_0 \in I $ then 
$ F' \left( \frac{x_0+y}{2} \right) =  
- 2 f_1'(x_0) $ for all $ y \in I $. Therefore 
$ \frac{1}{2}(x_0 + I) \in \mathscr{A}_F \,$, 
which is a contradiction. 
\end{proof}

So from now on, we may divide by the derivatives 
$ g_1' $ and $ g_2' $. 
From \eqref{eq-Invariance-eq}, 
we deduce (see \cite[Theorem 2.]{Kis22}) 
that the system of two functional equations 
\begin{equation}\label{eq-derivatives_system}
\begin{cases}
\varphi \left( \frac{x+y}{2} \right) 
\left( \psi_1(x) + \psi_1(y) \right) = 
\Psi_1(x) + \Psi_1(y) \\ 
\varphi \left( \frac{x+y}{2} \right) 
\left( \psi_2(x) - \psi_2(y) \right) = 
\Psi_2(x) - \Psi_2(y),  
\end{cases} 
\qquad x,y \in I 
\end{equation}
holds for the functions 
\begin{equation}\label{eq_def_auxfuncs}
\varphi := \frac{1}{2}F' \,, \qquad 
\psi_k := \frac{1}{g_1'} + (-1)^k \frac{1}{g_2'} \,, \qquad
\Psi_k := -\frac{f_1'}{g_1'} -(-1)^k \frac{f_2'}{g_2'} \,, 
\qquad k = 1,2.
\end{equation}
Furthermore, by putting $ y = x $,  
clearly $ \Psi_1 = \varphi \cdot \psi_1 $ 
in the first equation. 
Thus our aim is to solve the functional equation 
\begin{equation}\label{eq-auxiliary-sum}
\varphi \left( \frac{x+y}{2} \right) 
\bigl( \psi_1(x) + \psi_1(y) \bigr) = 
\varphi(x) \psi_1(x) + \varphi(y) \psi_1(y),  
\qquad x,y \in I, 
\end{equation}
while assuming that there exist differentiable 
functions $ F, g_1 \,, g_2 : I \map \RR $ 
defined on an open interval 
$ \emptyset \neq I \subseteq \RR $ such that 
$ 0 \notin g_1'(I) \cup g_2'(I) $ and 
$ \varphi = \frac{1}{2} F' $ and 
$ \psi_1 = \frac{1}{g_1'} - \frac{1}{g_2'} \,$, 
moreover $ \mathscr{A}_F = \emptyset $, 
i.~e. $ F $ is not affine on 
any nonempty open subinterval of $ I $. 
Obviously, this condition is equivalent to the fact that 
$ \varphi $ is not constant on any nonempty open subinterval 
of $ I $. This natural assumption will appear in 
Theorem \ref{thm-add_equation_contsol}. 

The most general setting in which equation 
\eqref{eq-auxiliary-sum} is solved is that 
$ \varphi $ is continuous while the set of zeros of $ \psi_1 $ 
fulfill a technical assumption which is weaker 
then continuity. The details can be found in \cite{KP18}. 
It turns out that we have two kinds of solutions. 
Either $ \varphi $ is constant on open subinterval of $ I $ 
and $ \psi_1 $ is zero on another subinterval of $ I $. 
Or $ \varphi $ is strictly monotone, 
$ \psi_1 $ is different from zero everywhere, 
moreover $ \varphi, \psi_1 $ 
are infinitely many times differentiable. 
In our context, solutions of the first type are limited 
to the case when $ \psi_1 $ is identically zero on $I$ 
(and $ \varphi $ is arbitrary). 

Our idea is to show that if $ \varphi, \psi_1 $ 
are constructed from derivatives 
having the properties mentioned above, and 
the pair $ (\varphi, \psi_1) $ 
is a solution of \eqref{eq-auxiliary-sum} 
then $ \psi_1 $ is identically $ 0 $ 
or $ \varphi $ and $ \psi_1 $ are both continuous functions. 
This would mean that \cite[Theorem 11.]{KP18} can 
be applied in order to characterize the solutions. 
The proof of our first main theorem in this section 
requires some 
fundamental but non-trivial 
observations related to measure theory. 
Therefore we formulate a lemma which will play an 
essential role during the proof. 
The Lebesgue measure on $ \RR $ will be denoted by $ \mu $. 

\begin{lemm}\label{lemm-positive_measure_trns}
Let $ A_1 \,, A_2 \subseteq \RR $ be two sets with 
positive Lebesgue measure and let $ D \subseteq \RR $ 
be a dense subset of $ \RR $. 
Then there exists $ d \in D $ such that 
\[
\mu \Bigl( \frac{A_1 + d}{2} \cap A_2 \Bigr) > 0. 
\]
\end{lemm}

\begin{proof}
As a consequence of the Lebesgue Density Theorem, there exist 
density points in $ A_1 $ and $ A_2 \,$. 
In particular, there exist 
$ a_k \in A_k $ and $ \varepsilon > 0 $ such that 
for the interval 
$ V_k := [a_k - \varepsilon , a_k + \varepsilon ] $ 
we have 
\[
\frac{\mu \left( V_k 
\cap A_k \right)}{2 \varepsilon} > \frac{3}{4} \,,
\qquad 
k=1,2. 
\]
Consequently, 
$ \mu \left( V_1 \setminus A_1 \right) < 
\frac{\varepsilon}{2} \,$. 
Since $ D $ is dense, there exists $ d \in D $ 
such that $ \vert \frac{a_1+d}{2} - a_2 \vert < 
\frac{\varepsilon}{4} \, $, hence 
$ \frac{V_1+d}{2} \subset V_2 $ is fulfilled. 
Some easy calculations show that 
\begin{align*}
\mu \left( \frac{A_1 + d}{2} \cap A_2 \right) 
& \geq
\mu \left( \frac{(A_1 \cap V_1) + d}{2} \cap A_2 \right) \\ 
& = 
\mu \left( \frac{V_1 + d}{2} \cap A_2 \right) - 
\mu \left( \frac{(V_1 \setminus A_1) + d}{2} \cap A_2 \right) \\ 
& = 
\mu \left( V_2 \cap A_2 \right) - 
\mu \left( \left( V_2 \setminus \frac{V_1 + d}{2} \right) 
\cap A_2 \right) - 
\mu \left( \frac{(V_1 \setminus A_1) + d}{2} \cap A_2 \right) \\ 
& \geq 
\mu \left( V_2 \cap A_2 \right) - 
\mu \left( V_2 \setminus \frac{V_1 + d}{2} \right) - 
\mu \left( \frac{(V_1 \setminus A_1) + d}{2} \right) \\ 
& > \frac{3 \varepsilon}{2} - \varepsilon - 
\frac{\varepsilon}{4} = \frac{\varepsilon}{4} > 0. 
\end{align*}
\end{proof}

If $ I \subseteq \RR $ is an interval, 
$ h : I \map \RR $ is a function then the subset 
where $ h $ does not vanish will be called the 
{\em support of $ h $} and we denote it by $ \supp \, h $: 
\[
\supp \, h = \lbrace x \in I : h(x) \neq 0 \rbrace. 
\]
The next proposition contains most of the preliminary 
results which will make it possible to solve 
equation \eqref{eq-auxiliary-sum} under reduced regularity. 

\begin{prop}\label{prop-derivative_sols}
Suppose that the functions 
$ \varphi , \psi_1 : I \map \RR $ 
solve the functional equation 
\[
\varphi \left( \frac{x+y}{2} \right) 
\bigl( \psi_1(x) + \psi_1(y) \bigr) = 
\varphi(x) \psi_1(x) + \varphi(y) \psi_1(y), 
\qquad x,y \in I. 
\]
Let us also suppose that there exist 
differentiable 
functions $ F, g_1 \,, g_2 : I \map \RR $ 
such that $ 0 < g_1'(x) \cdot g_2'(x) $ for all 
$ x \in I $ moreover 
$ \varphi = \frac{1}{2} F' $ and 
$ \psi_1 = \frac{1}{g_1'} - \frac{1}{g_2'} \,$. 

Then $ \varphi, \psi_1 $ have the following properties: 
\begin{enumerate}
\item[(i)] If $ \emptyset \neq L \subseteq I $ is an open 
interval, $ c \in \RR $ is an arbitrary number, and 
\[ 
\mu \bigl( \lbrace x \in L : \varphi(x) \neq c \rbrace \bigr) = 0 
\] 
then $ \varphi(x) = c $ for all $ x \in L $. 
\item[(ii)] If $ \emptyset \neq L \subseteq I $ is an open 
interval and 
$ \mu \left( \supp \, \psi_1 \cap L \right) = 0 $ 
then $ \psi_1(x) = 0 $ for all $ x \in L $. 
\item[(iii)] If $ x,y \in I $ such that $ \psi_1(x) > 0 $ and 
$ \psi_1(y) < 0 $ then there exists 
\[
z \in \left] \min(x,y) \, , \, \max(x,y) \right[ 
\mbox{ such that } \psi_1(z) = 0. 
\] 
\item[(iv)] Let $ x,y \in I $ such that 
$ \psi_1(x) \psi_1(y) > 0 $ and 
let $ J \subseteq \RR $ be an interval containing at least 
one point. Suppose that 
$ \varphi(x), \varphi(y) \in J $. 
Then 
$ \varphi \left( \frac{x+y}{2} \right) \in J $ as well. 
\end{enumerate}

\end{prop}

\begin{proof} 
{\em Part (i):} 
It is well known, that if the derivative of a differentiable 
function is zero almost everywhere then the derivative 
is zero everywhere (cf. \cite[Theorem 7.21.]{Rud87}). 
The assumption in {\em (i)} is that the 
derivative of the function 
$ \tilde{F}(x) := \frac{1}{2} F(x) - cx $ 
is zero almost everywhere on $ L $. But then 
$ \tilde{F}'(x) = 0 $ for all $ x \in L $. That is, 
$ \varphi(x) = c $ for all $ x \in L $. 

{\em Part (ii):} Observe that $ \psi_1(x) = 0 $ means 
$ \frac{1}{g_1'(x)} - \frac{1}{g_2'(x)} = 0 $. 
Multiplying by $ g_1'(x) g_2'(x) $ this is equivalent to 
\[
0 = g_2'(x) - g_1'(x) = \left( g_2 - g_1 \right)'(x). 
\]
Thus the assumption of {\em (ii)} means that the derivative 
$ \left( g_2 - g_1 \right)' $ is zero almost everywhere 
on $ L $. As before, this implies $ g_2'(x) = g_1'(x) $ 
for every $ x \in L $, so 
$ \psi_1(x) = 0 $ for every $ x \in L $. 

{\em Part (iii):} Now  $ \psi_1(x) > 0 $ means 
$ \frac{1}{g_1'(x)} - \frac{1}{g_2'(x)} > 0 $. 
Multiplying by $ g_1'(x) g_2'(x) > 0 $ 
this is equivalent to 
\[
0 < g_2'(x) - g_1'(x) = \left( g_2 - g_1 \right)'(x). 
\]
Similarly, $ \psi_1(y) < 0 $ is equivalent to 
\[
0 > g_2'(y) - g_1'(y) = \left( g_2 - g_1 \right)'(y). 
\]
However, a derivative is a Darboux function. In particular, 
there exists 
\[ 
z \in \left] \, \min(x,y) \, , \, \max(x,y) \, \right[ 
\] 
such that 
$ \left( g_2 - g_1 \right)'(z) = 0 $, 
but once again, this is equivalent to 
$ \psi_1 (z) = 0 $. 

{\em Part (iv):} 
Clearly, it is enough to prove the statement for 
intervals of the following type: 
\[ 
J \in \big\lbrace \, 
\left] -\infty , \alpha \right[ \,,\, 
\left] -\infty , \alpha \right] \,,\, 
\left] \alpha , +\infty \right[ \,,\, 
\left[ \alpha , +\infty \right[ \, 
\big\rbrace , 
\] 
where $ \alpha $ is a given real number. 
Indeed, every interval $ J \neq \RR $ 
can be written as the intersection 
of two intervals of the above kind, while for $ J = \RR $ 
the statement is trivial. 
We detail the proof only for 
the first case, the other three are analogous. 
Assume that $ \varphi(x) < \alpha $ and 
$  \varphi(y) < \alpha $. 
Now we have two possibilities. Either $ \psi_1(x), \psi_1(y) $ are 
both positive, or they are both negative. 
If $ \psi_1(x) > 0 $, $ \psi_1(y) > 0 $ then 
$  \varphi(x) \psi_1(x) < \alpha \psi_1(x) $ and 
$  \varphi(y) \psi_1(y) < \alpha \psi_1(y) $. 
Hence 
\begin{align*}
\varphi \left( \frac{x+y}{2} \right) 
\bigl( \psi_1(x) + \psi_1(y) \bigr) & = 
\varphi(x) \psi_1(x) +  \varphi(y) \psi_1(y) \\ 
& < \alpha \psi_1(x) + \alpha \psi_1(y) 
= \alpha \left( \psi_1(x) + \psi_1(y) \right). 
\end{align*}
Dividing by the positive number 
$ \left( \psi_1(x) + \psi_1(y) \right) $ we get that 
$ \varphi \left( \frac{x+y}{2} \right) < \alpha $. 

In the other case, when 
$ \psi_1(x) < 0 $, $ \psi_1(y) < 0 $, we have 
$ \varphi(x) \psi_1(x) > \alpha \psi_1(x) $ and 
$ \varphi(y) \psi_1(y) > \alpha \psi_1(y) $. 
Hence 
\begin{align*}
\varphi \left( \frac{x+y}{2} \right) 
\bigl( \psi_1(x) + \psi_1(y) \bigr) & = 
\varphi(x) \psi_1(x) +  \varphi(y) \psi_1(y) \\ 
& > \alpha \psi_1(x) + \alpha \psi_1(y) 
= \alpha \left( \psi_1(x) + \psi_1(y) \right). 
\end{align*}
Dividing by the negative number 
$ \left( \psi_1(x) + \psi_1(y) \right) $ 
we again arrive to 
$ \varphi \left( \frac{x+y}{2} \right) < \alpha $. 
So the proof is done for 
$ J = \left] -\infty , \alpha \right[ \,$. Following 
the previous remarks, the completion of the proof for 
the general claim is left to the reader. 
\end{proof}

Our most important result of this section 
is the following statement. 

\begin{thm}\label{thm-add_equation_contsol}
Suppose that the functions 
$ \varphi, \psi_1 : I \map \RR $ 
solve the functional equation 
\[
\varphi \left( \frac{x+y}{2} \right) 
\bigl( \psi_1(x) + \psi_1(y) \bigr) = 
\varphi(x) \psi_1(x) + \varphi(y) \psi_1(y), 
\qquad x,y \in I. 
\]
Let us suppose that there exist 
differentiable 
functions $ F, g_1 \,, g_2 : I \map \RR $ 
such that $ 0 < g_1'(x) \cdot g_2'(x) $ for all 
$ x \in I $ moreover 
$ \varphi = \frac{1}{2} F' $ and 
$ \psi_1 = \frac{1}{g_1'} - \frac{1}{g_2'} \,$. 
Let us also assume that there is no nonempty 
open interval $ L \subseteq I $ such that 
$ \varphi $ is constant on $ L $. 

Then either $ \psi_1 $ is identically zero, or 
$ 0 \notin \psi_1(I) $ and 
$ \varphi, \psi_1 $ are continuous functions.     
\end{thm}

\begin{proof}
In order to make the proof easier to comprehend we 
divide the reasoning into four parts. 

{\em Part 1.)} 
In the first step we prove that if there exists 
an open interval $ \emptyset \neq L \subset I $ 
such that $ \psi_1(x) = 0 $ for all $ x \in L $ then 
$ \psi_1(x) = 0 $ must hold for all $ x \in I $ as well. 
Otherwise, there would exist $ y_0 \in I \setminus L $ 
such that $ \psi_1(y_0) \neq 0 $. Then, for all $ x \in L $, 
equation \eqref{eq-auxiliary-sum} yields 
\[
\varphi \left( \frac{x+y_0}{2} \right) 
\psi_1(y_0) = \varphi(y_0) \psi_1(y_0). 
\]
Here we can divide by $ \psi_1(y_0) \neq 0 $ 
and get that 
$ \varphi \left( \frac{x+y_0}{2} \right) = \varphi(y_0) $. 
That is, $ \varphi $ would be constant on the interval 
$ \frac{L + y_0}{2} $ which is excluded by the assumptions. 
Thus $ \psi_1(x) = 0 $ for all $ x \in I $. 

{\em Part 2.)} 
Therefore, from now on, let us consider 
the case that $ \supp \, \psi_1 $ is dense in $ I $. 
In fact, much more is true. Let us observe that for any 
open interval $ \emptyset \neq L \subseteq I $ 
the set $ \supp \, \psi_1 \cap L $ 
has positive measure, that is, 
\[
\mu \bigl( \lbrace x \in L : 
\psi_1(x) \neq 0 \rbrace \bigr) > 0. 
\] 
Otherwise, assertion {\em (ii)} of 
Proposition \ref{prop-derivative_sols}, 
would imply that $ \psi_1 \vert _L \equiv 0 $ contradicting 
the fact that $ \supp \, \psi_1 $ is dense. 

{\em Part 3.)} 
In the next part we will prove that 
if $ \psi_1 $ is not identically zero then 
the set of zeros of $ \psi_1 $ cannot be dense in $ I $. 
Let us assume the contrary, namely that 
\[
Z := \lbrace x \in I : \psi_1(x) = 0 \rbrace 
\mbox{ is dense in } I. 
\]
As we have seen in {\em Part 1}, for all $ x \in Z $ 
and for all $ y \in \supp \, \psi_1 $ 
we have 
\begin{equation}\label{eq-one_term_gzero}
\varphi \left( \frac{x+y}{2} \right) 
\psi_1(y) = \varphi(y) \psi_1(y). 
\end{equation}
If we divide by $ \psi_1(y) \neq 0 $ , we get that 
$ \varphi $ is constant on the set 
$ \frac{Z + y}{2} $ with value $ \varphi(y) $. 
Now we distinguish two cases: \\ 
{\em Case 1.} 
Suppose that $ \mu (Z) > 0 $. 
Consider a partition of $ I $ consisting of 
pairwise disjoint intervals such that all of them 
has the same length $ \gamma < \frac{\mu (Z)}{3} \,$. 
Because of the $ \sigma $-additivity of the measure, 
there is an interval $ L $ in this partition such that 
$ \inf I < \inf L =: \alpha $, 
$ \beta := \sup L < \sup I $ and 
$ \mu \left( L \cap Z \right) > 0 $ (otherwise 
$ \mu (Z) \leq \frac{2 \mu (Z)}{3} $ would hold). 
The conditions for $ \alpha, \beta $ imply that 
the inclusion 
\[
L_0 := \left] \, 2\alpha - \beta 
, 2\beta - \alpha \, \right[ \subset I 
\]
holds as well. 
Furthermore, in {\em Part 2} of the proof we 
have established that at least one of the sets 
\[
N = \lbrace x \in L : 
\psi_1(x) < 0 \rbrace 
\ \mbox{ and } \ 
P = \lbrace x \in L : 
\psi_1(x) > 0 \rbrace 
\]
has positive measure. We consider the case 
$ \mu (N) > 0 $, while $ \mu (P) > 0 $ can be 
handled analogously. 
Now let us apply 
Lemma \ref{lemm-positive_measure_trns} 
for the sets $ L \cap Z $ and $ N $ which have 
positive measure, and for the dense set 
$ D := ( \supp \, \psi_1 \cap L_0) 
\cup \left( \RR \setminus L_0 \right) $. 
According to that, there exists 
$ y_0 \in D $ such that 
\[
\mu \left( \frac{(L \cap Z) + y_0}{2} \cap N \right) > 0. 
\] 
Since $ L $ is the middle third of the interval $ L_0 $, 
it is clear that $ y_0 \in \supp \, \psi_1 \cap L_0 \, $. 
By equation \eqref{eq-one_term_gzero}, 
we have that $ \varphi $ is constant with value 
$ \varphi(y_0) $ 
on the positive measure set 
$ C_0 = \frac{Z + y_0}{2} \cap N $. 
Now let us apply assertion {\em (iv)} of 
Proposition \ref{prop-derivative_sols} with 
the singleton interval $ J = \lbrace \varphi(y_0) \rbrace $. 
Consequently, $ \varphi $ is constant with value 
$ \varphi(y_0) $ on the set $ \frac{C_0 + C_0}{2} $. 
But the Theorem of Steinhaus \cite{Ste20} claims that 
$ \frac{C_0 + C_0}{2} $ contains an open interval. 
So $ \varphi $ is constant on an open interval, however 
this contradicts the assumptions of our theorem. \\ 
{\em Case 2.} 
Thus in the second case we have to consider 
$ \mu (Z) = 0 $. 
Let $ \emptyset \neq J \subseteq \RR $ 
be an arbitrary interval such that 
$ \mu \left( \varphi^{-1}(J) \right) > 0 $. 
Similarly as in the previous case, there exists 
a nonempty open interval 
$ \left] \, \alpha , \beta \, \right[ = L \subset I $ 
such that $ \mu \left( \varphi^{-1}(J) \cap L \right) > 0 $ 
and 
\[
L_0 = \left] \, 2\alpha - \beta 
, 2\beta - \alpha \, \right[ \subset I. 
\]
What is more, since $ \mu (Z) = 0 $, at least one of the sets 
\[
N_J := \varphi^{-1}(J) \cap L \cap N 
\ \mbox{ and } \ 
P_J := \varphi^{-1}(J) \cap L \cap P 
\]
has positive measure where 
$ N = \lbrace x \in L : \psi_1(x) < 0 \rbrace $ 
and 
$ P = \lbrace x \in L : \psi_1(x) > 0 \rbrace $. 
We discuss the first possibility, when 
$ \mu \bigl( \varphi^{-1}(J) \cap L \cap N  \bigr) > 0 $, 
the other one is analogous. 
Just as in the previous case, 
applying assertion {\em (iv)} of 
Proposition \ref{prop-derivative_sols} for $ J $, 
we get that $ \frac{N_J + N_J}{2} \subseteq \varphi^{-1}(J) $. 
Due to the Theorem of Steinhaus, 
there exists a nonempty open interval 
$ M \subset \frac{N_J + N_J}{2} 
\subseteq \varphi^{-1}(J) \cap L $. 

Let us emphasize that $ J \subseteq \RR $ 
was arbitrary. 
Hence, if $ \widetilde{J} \subseteq \RR $ 
is another interval such that 
$ \widetilde{J} \cap J = \emptyset $ but 
$ \mu \left( \varphi^{-1}(\widetilde{J}) \cap L \right) > 0 $ 
then there exists another nonempty open interval 
$ \widetilde{M} \subseteq 
\varphi^{-1}(\widetilde{J}) \cap L $. 
Obviously, $ M \cap \widetilde{M} = \emptyset $. 
Since $ Z $ and $ \supp \, \psi_1 $ are dense in $ I $, 
there exists $ x \in Z \cap L_0 $ and 
$ y \in \supp \, \psi_1 \cap M $ such that 
$ \frac{x + y}{2} \in \widetilde{M} $. 
Equation \eqref{eq-one_term_gzero} yields 
\[
\varphi \left( \frac{x + y}{2} \right) = \varphi(y) \in J. 
\]
That is, 
$ \frac{x + y}{2} \in M \cap \widetilde{M} = \emptyset $, 
a contradiction. Summarizing these observations, two 
disjoint intervals cannot have preimages in $ L $ 
with positive measure simultaneously. 
In particular, for every $ a \in \RR $, 
the measure of the set 
$ \varphi^{-1} \vert _L 
\left( \, \left] - \infty , a \right] \, \right) $ 
is either $ 0 $ or $ \mu(L) $. 
Defining the real number 
\[
c := \inf \lbrace a \in \RR : 
\mu \bigl( 
\varphi^{-1} \vert _L 
\left( \, \left] - \infty , a \right] \, \right) 
\bigr) > 0 \rbrace 
\] 
it is clear that 
$ \mu \bigl( \varphi^{-1} \vert _L \left( c \right) \bigr) 
= \mu (L) $. 
According to assertion {\em (i)} of 
Proposition \ref{prop-derivative_sols} 
this implies that $ \varphi $ is constant on the whole 
interval $ L $ which is again impossible. 

This contradiction concludes the third part of the proof. 

{\em Part 4.)} 
In the final part we deal with the only remaining 
situation, namely when $ Z $ is not dense. 
Let $ K $ denote an arbitrary nonempty 
open subinterval of $ \supp \, \psi_1 $. According to 
claim {\em (iii)} of Proposition \ref{prop-derivative_sols}, 
$ \psi_1 $ does not change its sign on $ K $. 
Suppose that $ v := \varphi(x) = \varphi(y) $ 
for some points 
$ x,y \in K $, $ x \neq y $. Then we shall apply 
claim {\em (iv)} of Proposition \ref{prop-derivative_sols} 
for $ J = \lbrace v \rbrace $ in order to see that 
$ \varphi^{-1}(v) $ is closed under taking the arithmetic mean, 
so it must be dense in $ \left[ x,y \right] $. 
On the other hand, at least one of the sets 
\[
V_{+} := \lbrace \, t \in \left[ x,y \right] : 
\varphi(t) > v \, \rbrace 
\ \mbox{ and } \ 
V_{-} := \lbrace \, t \in \left[ x,y \right] : 
\varphi(t) < v \, \rbrace 
\]
must have positive measure, otherwise $ \varphi $ would be 
constant on $ \left] \, x,y \, \right[ $ (see {\em (i)} 
in Proposition \ref{prop-derivative_sols}). 
For instance, if $ \mu \left( V_{+} \right) > 0 $ then 
claim {\em (iv)} of Proposition \ref{prop-derivative_sols} 
implies that $ \frac{V_{+} + V_{+}}{2} \subseteq V_{+} $. 
Due to Steinhaus' Theorem, this average set contains an interval. 
But this contradicts the fact that $ \varphi^{-1}(v) $ 
is dense in $ \left[ x,y \right] $. 
This means that $ \varphi $ is injective on $ K $. 

In the next step let us check that $ \psi_1 $ is 
nowhere zero. Let us suppose its contrary, 
namely that $ K \neq I $ 
is a maximal open interval in $ \supp \, \psi_1 $. 
Since $ K \neq I $, there exists 
$ z \in (I \setminus K) \cap Z $ and $ y \in K $ 
such that $ \frac{z+y}{2} \in K $. 
Now equation \eqref{eq-one_term_gzero} 
implies $ \varphi \left( \frac{z+y}{2} \right) = \varphi(y) $ 
which contradicts the injectivity of $ \varphi $ on $ K $. 
Thus $ I = K $, so $ \supp \, \psi_1 = I $ and $ \varphi $ 
is injective. 

Using that $ \varphi $ is a derivative, it follows that 
$ \varphi $ is a Darboux function. Together with the 
injectivity this implies that $ \varphi $ is continuous 
(see \cite{GF66}) 
and strictly monotone. 
Finally, after fixing an arbitrary $ y_0 \in I $, 
rearranging equation \eqref{eq-auxiliary-sum} 
and expressing $ \psi_1 $ we obtain 
\[
\psi_1(x) = 
\frac{\varphi(y_0)\psi_1(y_0) - 
\varphi \left( \frac{x+y_0}{2} \right) \psi_1(y_0)}{
\varphi \left( \frac{x+y_0}{2} \right) - \varphi(x)} 
\]
for all $ x \in I $. 
The strict monotonicity of $ \varphi $ ensures that this formula 
indeed makes sense. Moreover, the continuity of $ \varphi $ 
clearly implies that $ \psi_1 $ is continuous as well. 
\end{proof}

This theorem makes possible to determine 
the solutions of the 
system of functional equations \eqref{eq-derivatives_system}. 
With the help of Theorem \ref{thm-add_equation_contsol}
we are able to utilize the results in the literature 
concerning the equations of 
\eqref{eq-derivatives_system}
under the weaker regularity assumptions in our setting. 

\begin{thm}\label{thm-sols_of_derivative_system}
Let 
$ F, f_1 \,, f_2 \,, g_1 \,, g_2 : I \map \RR $ 
be differentiable functions such that 
$ g_1'(x) \cdot g_2'(x) > 0 $ for every $ x \in I $ 
and $ F $ is not affine on any open subinterval 
$ \emptyset \neq L \subseteq I $. 
Suppose that the functions 
\[
\varphi := \frac{1}{2}F' \,, \qquad 
\psi_k := \frac{1}{g_1'} + (-1)^k \frac{1}{g_2'} \,, \qquad
\Psi_k := -\frac{f_1'}{g_1'} -(-1)^k \frac{f_2'}{g_2'} \,, 
\qquad k = 1,2
\]
solve the system of functional equations 
\eqref{eq-derivatives_system}, namely 
\begin{equation*}
\begin{cases}
\varphi \left( \frac{x+y}{2} \right) 
\left( \psi_1(x) + \psi_1(y) \right) = 
\Psi_1(x) + \Psi_1(y) \\ 
\varphi \left( \frac{x+y}{2} \right) 
\left( \psi_2(x) - \psi_2(y) \right) = 
\Psi_2(x) - \Psi_2(y), 
\end{cases} 
\qquad x,y \in I. 
\end{equation*} 
Then $ \Psi_1 = \varphi \cdot \psi_1 $, moreover 
there exist constants 
$ a , b , c , d , \gamma , \lambda, \nu \in \RR $ 
such that $ ad - bc \neq 0 $ and exactly 
one of the following seven cases holds: 

\begin{enumerate}
\item[\textit{(1.1)}] $ \gamma < 0 $ and for all $ x \in I $ 
\begin{alignat*}{3}
\varphi(x) &= 
\frac{c \sin(\kappa x) + d \cos(\kappa x)}{a \sin(\kappa x) 
+ b \cos(\kappa x)} \, , \qquad 
& \psi_1(x) &= a \sin(\kappa x) + b \cos(\kappa x) \neq 0 , \\ 
\psi_2(x) &= -a \cos(\kappa x) + b \sin(\kappa x) + \lambda 
\, , \qquad 
& \Psi_2(x) &= -c \cos(\kappa x) + d \sin(\kappa x) + \nu; 
\mbox{ or } 
\end{alignat*} 
\item[\textit{(1.2)}] $ \gamma = 0 $ and for all $ x \in I $ 
\begin{alignat*}{3}
\varphi(x) & = 
\frac{cx + d}{ax + b} \, , \qquad 
& \psi_1(x) &= ax + b \neq 0 , \\ 
\psi_2(x) &= \frac{1}{2} a x^2 + b x + \lambda 
\, , \qquad 
& \Psi_2(x) &= \frac{1}{2} c x^2 + d x + \nu ; 
\mbox{ or } 
\end{alignat*}
\item[\textit{(1.3)}] 
$ \gamma > 0 $ and for all $ x \in I $ 
\begin{alignat*}{3}
& \varphi(x) = 
\frac{c \sinh(\kappa x) + d \cosh(\kappa x)}{a \sinh(\kappa x) 
+ b \cosh(\kappa x)} \, , \qquad 
& \psi_1(x) &= a \sinh(\kappa x) + b \cosh(\kappa x) \neq 0 , \\ 
& \psi_2(x) = a \cosh(\kappa x) + b \sinh(\kappa x) + \lambda 
\, , \qquad 
& \Psi_2(x) &= c \cosh(\kappa x) + d \sinh(\kappa x) + \nu; 
\mbox{ or } 
\end{alignat*}  
\item[\textit{(2)}]
$ \psi_1(x) = \Psi_1(x) = 0 $ , 
$ \psi_2 (x) = a \neq 0 $ and $ \Psi_2 (x) = b $ 
for all $ x \in I $, moreover $ \varphi $ is arbitrary; or 
\item[\textit{(3.1)}] 
$ \gamma < 0 $ and for all $ x \in I $  
\begin{alignat*}{3} 
& \varphi(x) = 
\frac{c \sin(\kappa x) + d \cos(\kappa x)}{a \sin(\kappa x) 
+ b \cos(\kappa x)} \, , \qquad 
& \psi_1(x) & = \Psi_1(x) = 0 , \\ 
& \psi_2(x) = 
-a \cos(\kappa x) + b \sin(\kappa x) + \lambda \, , \qquad 
& \Psi_2(x) & = 
-c \cos(\kappa x) + d \sin(\kappa x) + \nu;  
\mbox{ or } 
\end{alignat*} 
\item[\textit{(3.2)}] 
$ \gamma = 0 $ and for all $ x \in I $  
\begin{alignat*}{3}
& \varphi(x) = 
\frac{cx + d}{ax + b} \, , \qquad 
& \psi_1(x) & = \Psi_1(x) = 0 , \\ 
& \psi_2(x) = \frac{1}{2} a x^2 + b x + \lambda \, , \qquad 
& \Psi_2(x) & = \frac{1}{2} c x^2 + d x + \nu ; 
\mbox{ or } 
\end{alignat*} 
\item[\textit{(3.3)}] 
$ \gamma > 0 $ and for all $ x \in I $  
\begin{alignat*}{3} 
& \varphi(x) = 
\frac{c \sinh(\kappa x) + d \cosh(\kappa x)}{a \sinh(\kappa x) 
+ b \cosh(\kappa x)} \, , \qquad 
& \psi_1(x) & = \Psi_1(x) = 0 , \\ 
& \psi_2(x) = 
a \cosh(\kappa x) + b \sinh(\kappa x) + \lambda \, , \qquad 
& \Psi_2(x) & = 
c \cosh(\kappa x) + d \sinh(\kappa x) + \nu; 
\end{alignat*} 

\end{enumerate}
where $ \kappa = \sqrt{ \vert \gamma \vert } $. 
Conversely, 
if $ \Psi_1 = \varphi \cdot \psi_1 $ then 
in each of the above cases 
the functions $ \varphi, \psi_k \,, \Psi_k $ 
solve the system \eqref{eq-derivatives_system}. 
\end{thm}

\begin{proof} 
First of all, $ \Psi_1 = \varphi \cdot \psi_1 $ 
follows from the first equation after putting $ x = y $. 
Let us observe that for 
$ \varphi $ and $ \psi_1 $ 
the assumptions of Theorem \ref{thm-add_equation_contsol} 
are fulfilled. 
Hence either $ \psi_1 $ is identically zero 
or $ \psi_1 $ does not vanish anywhere 
and $ \varphi , \psi_1 $ are continuous functions. 

{\em The case of $ 0 \notin \psi_1(I) $.} 
Here the exact form of $ \varphi $ and $ \psi_1 $, 
appearing in cases {\em (1.1)--(1.3)}, 
follows from \cite[Theorem 4]{Kis22}. 
Then, due to the continuity of $ \varphi $, 
\cite[Theorem 5]{Kis22} can be applied for the 
second equation of \eqref{eq-derivatives_system}. 
This implies the formulas for $ \psi_2 $ and $ \Psi_2 $ 
in cases {\em (1.1)--(1.3)}. 
The coefficients appearing in these formulas must 
be the same $ a,b,c,d,\gamma,\kappa $ values as 
in the formula for $ \varphi $ because 
$ \varphi $ is the common function in the two 
equations of \eqref{eq-derivatives_system}. 

{\em The case of $ \psi_1 = 0 $.} 
This assertion means 
$ \frac{1}{g_1'(x)} - \frac{1}{g_2'(x)} = 0 $ 
for all $ x \in I $. That is, 
$ g_1'(x) = g_2'(x) $, hence there exists $ c_1 \in \RR $ 
such that $ g_2(x) = g_1(x) + c_1 $ for all $ x \in I $. 
Thus the main equation \eqref{eq-Invariance-eq} 
can be written in the form 
\[
F \left( \frac{x+y}{2} \right) + f_1(x) + f_2(y) = 
G \left( g_1(x) + g_1(y) + c_1 \right) , 
\qquad x,y\in I. 
\]
Swapping the variables $x$ and $y$ and subtracting the 
two equations we arrive to 
\[ 
F \left( \frac{x+y}{2} \right) + f_1(x) + f_2(y) = 
F \left( \frac{x+y}{2} \right) + f_1(y) + f_2(x) , 
\qquad 
x,y\in I. 
\]
Consequently, $ f_1 - f_2 $ is constant, so there exists 
$ c_2 \in \RR $ such that 
$ f_2(x) = f_1(x) + c_2 $ for all $ x \in I $. 
In particular, we have 
\[
\varphi := \frac{1}{2}F' \,, \qquad 
\psi_2 := \frac{2}{g_1'} \,, \qquad
\Psi_2 := -\frac{2 f_1'}{g_1'} 
\hspace{4mm} \mbox{ and } \hspace{4mm} 
\psi_1=\Psi_1 = 0. 
\]
Thus we only need to focus on the second equation of 
the system \eqref{eq-derivatives_system}. 
We are going to utilize one of the key observations 
of \cite{KP19} and transform this equation in a specific manner. 

For any $ h > 0 $ let us define the set 
\[
I_h := \lbrace x \in I : 
x-h \in I \mbox{ and } x+h \in I \rbrace. 
\]
Moreover, let 
$ H_0 := \lbrace h > 0 : I_h \neq \emptyset \rbrace $. 
According to \cite[Theorem 2]{KP19}, for all $ h \in H_0 $, 
the functions $ \varphi \vert _{I_h} $ and 
$ \psi_2^{[h]} : I_h \map \RR $ defined by 
\[ 
\psi_2^{[h]}(x) := \psi_2(x+h) - \psi_2(x-h),  
\qquad x \in I_h 
\] 
solve the functional equation 
\begin{equation}\label{eq-addition_h}
\varphi \left( \frac{x+y}{2} \right) 
\left( \psi_2^{[h]}(x) + \psi_2^{[h]}(y) \right) = 
\varphi(x) \psi_2^{[h]}(x) + \varphi(y) \psi_2^{[h]}(y) , 
\qquad x,y \in I_h \,. 
\end{equation}
Let us observe that the assumptions of 
our Theorem \ref{thm-add_equation_contsol} are 
fulfilled for $ \varphi $ and $ \psi_2^{[h]} $. 
For $ \varphi $ this is obvious, while on the other hand, 
$ \psi_2^{[h]} $ is indeed the 
difference of the reciprocals of two appropriate derivatives. 
More precisely, for any $ h \in H_0 $ 
let us consider the functions 
$ g_1^{[h]} \,, g_2^{[h]} : I_h \map \RR $ defined by 
\[
g_1^{[h]}(x) := \frac{1}{2}g_1(x+h) \ \mbox{ and } \ 
g_2^{[h]}(x) := \frac{1}{2}g_1(x-h) \, 
\qquad x \in I_h \,. 
\] 
Then 
$ \left( g_1^{[h]} \right) ^{\prime} (x) \cdot 
\left( g_2^{[h]} \right) ^{\prime} (x) > 0 $ 
for all $ x \in I_h $, 
as $ g_1' $ is nowhere zero, so it has the same sign everywhere. 
Moreover 
\[
\psi_2^{[h]}(x) = 
\frac{2}{g_1'(x+h)} - \frac{2}{g_1'(x-h)} = 
\frac{1}{\left( g_1^{[h]} \right) ^{\prime} (x)} - 
\frac{1}{\left( g_2^{[h]} \right) ^{\prime} (x)} \,, 
\qquad x \in I_h \,. 
\]
Hence we can now apply Theorem \ref{thm-add_equation_contsol}. 
Firstly, for any fixed $ h \in H_0 \,$, 
the function $ \psi_2^{[h]} $ is 
either identically zero on $ I_h $ 
or nowhere zero on $ I_h \, $. 

{\em Case 1.} There exists $ \varepsilon \in H_0 $ such that 
for all $ 0 < h < \varepsilon $ the function 
$ \psi_2^{[h]} $ is identically zero on $ I_{h} \,$. 
We will show that in this case $ g_1' $ is a constant function. 
For this purpose let $ x, y \in I $ be two arbitrary points 
such that $ x < y $. 
Let us pick a positive integer $ m \in \NN $ such that 
\[
h_0 := \frac{y-x}{2m} < \varepsilon 
\ \mbox{would hold.}
\]
Let us consider the points 
\[ 
z_k := x + (2k-1) \cdot h_0 
\hspace{10mm} 
\mbox{for all } k = 1, 2, \dots , m .   
\]
Then $ z_1 = x + h_0 $ and $ z_m = y - h_0 $. Thus 
$ z_1 \,, z_m \in I_{h_0} $ hence $ z_k \in I_{h_0} $ 
holds as well for every $ k = 1, \dots , m $. 
Therefore 
\[
0 = \psi_2^{[h_0]}(z_k) = 
\frac{1}{g_1'(z_k + h_0)} - \frac{1}{g_1'(z_k - h_0)} 
\hspace{5mm} 
\mbox{for all } k = 1, 2, \dots , m . 
\] 
Equivalently, 
\[
g_1' \left( x + (2k-2) h_0 \right) = 
g_1' \left( z_k - h_0 \right) = 
g_1' \left( z_k + h_0 \right) = 
g_1' \left( x + 2k h_0 \right) 
\]
is true for all $ k = 1 , \dots , m $. 
This means 
\[
g_1'(x) = g_1'(x + 0h_0) = g_1'(x+ 2h_0) 
= \dots = 
g_1'(x + (2m-2)h_0) = g_1'(x + 2mh_0) = g_1'(y). 
\] 
But $ x < y $ were arbitrary points of $ I $, 
so $ g_1' $ attains the same value everywhere. 
Hence $ g_1' $ is a constant function, so there 
exists $ a \in \RR $ such that 
$ \psi_2 =  \frac{2}{g_1'} = a $. 
Hence $ \Psi_2(x) = \Psi_2 (y) $ 
for all $ x,y \in I $ and therefore 
there exists $ b \in \RR $ such that $ \Psi_2 = b $. 
This yields the possible solution 
{\em (2)} in our theorem. 

{\em Case 2.} From now on, we may assume that 
$ \psi_2 $ is not a constant function. 
When the previous case does not hold, 
there exists a sequence of positive numbers 
$ (h_n) : \NN \map H_0 $ 
such that $ h_n \to 0 $ and, 
for all $ n \in \NN $, 
the function $ \psi_2^{[h_n]} $ is not identically zero on 
the open interval $ I_{h_n} \,$. 
Applying Theorem \ref{thm-add_equation_contsol} 
for the equation \eqref{eq-addition_h} 
we can instantly establish that $ \psi_2^{[h_n]} $ 
is nowhere zero on $ I_{h_n} \,$, moreover 
$ \varphi \vert_{I_{h_n}} $ and 
$ \psi_2^{[h_n]} : I_{h_n} \map \RR $ 
are continuous functions. 
Using that $ h_n \to 0 $, we shall immediately conclude 
that $ \varphi $ is continuous on the entire interval 
\[
I = \bigcup_{n \in \NN} I_{h_n} \,. 
\]
In other words, $ \varphi = \frac{1}{2} F' $ 
is a continuous function. Hence the solutions of 
the second equation of \eqref{eq-derivatives_system} can be 
determined applying \cite[Theorem 5]{Kis22}. 
According to that, three different situations may occur 
(since the possibility of $ \psi_2 $ being constant on $I$ or 
$ \varphi $ being constant on an open subinterval is excluded). 
These solutions, described in  \cite[Theorem 5]{Kis22}, 
are listed in cases {\em (3.1)--(3.3)} of our theorem. 

In order to see that in each of the seven cases the 
appearing functions are indeed solutions of 
\eqref{eq-derivatives_system}, let us first observe that 
in case {\em (2)} both equations are trivially satisfied. 
For cases {\em (1.1)--(1.3)} we may use the sufficiency 
parts of \cite[Theorems 4 and 5]{Kis22}, while 
for cases {\em (3.1)--(1.3)} we may use the sufficiency 
part of \cite[Theorem 5]{Kis22} to check that these are indeed 
solutions. 
\end{proof} 

\begin{remark}
{\em Determining the unknown functions 
$ F, f_k \,, G , g_k $ in \eqref{eq-Invariance-eq}}. 
\\ \noindent 
At this point the solutions of the system of 
auxiliary equations \eqref{eq-derivatives_system} 
are completely characterized. 
From the definitions \eqref{eq_def_auxfuncs} of the functions 
$ \varphi, \psi_k \,, \Psi_k $ it is easy to express 
$ F', g'_k \,, f'_k \,$, using 
$ \psi_2 - (-1)^k \psi_1 = \frac{2}{g_k'} \neq 0 $. 
Namely, 
\[
F' = 2 \varphi, \qquad 
g_k' = \frac{2}{\psi_2 - (-1)^k \psi_1} \,, \qquad 
f_k' = - \frac{\Psi_2 - (-1)^k \Psi_1}{\psi_2 - (-1)^k \psi_1} \,,  
\qquad k=1,2.  
\]
Calculating $ F, g_k $ and $ f_k $ is now a matter of 
integration. When it is done, the only remaining unknown 
function $ G $ can be expressed from the main equation 
\eqref{eq-Invariance-eq}. 

This process was carried out in 
Section 4 of the paper \cite{Kis22} 
where the so-called regular solutions of 
\eqref{eq-Invariance-eq} are discussed. 
Our objective here is to conclude that the solutions 
of \eqref{eq-Invariance-eq} under our 
simple differentiability condition coincide with 
the class of solutions obtained in \cite{Kis22} 
under stronger assumptions. 
A regular solution in that setting means that 
$ F, f_k \,, G, g_k $ are continuously differentiable and 
the auxiliary functions defined in 
\eqref{eq_def_auxfuncs} solve \eqref{eq-derivatives_system}. 

The calculations for deriving 
$ F, f_k \,, G, g_k $ from $ \varphi, \psi_k \,, \Psi_k $ 
are rather lengthy and they are 
predominantly based on identities concerning trigonometric 
and hyperbolic functions and their integrals. 
Considering that the list of final solutions 
$ \left( F, f_k \,, G, g_k \right) $ is several pages 
long in itself (without the proofs), we do not intend 
to repeat them here. But for the sake of completeness, 
we provide the references and thoroughly check that 
their assumptions are fulfilled. 
\begin{itemize}
\item If the solution of the auxiliary system of equations 
\eqref{eq-derivatives_system} has the form {\em (2)} in 
Theorem \ref{thm-sols_of_derivative_system}, then 
$ g_1' = g_2' $ is constant and 
$ f_1' = f_2' $ is constant as well. Now the 
solutions of \eqref{eq-Invariance-eq} can be 
described applying 
Proposition \ref{gkc} for the interval $ J = I $. 
Let us recall that in this situation $ F $ can be arbitrary 
(however, of course, 
we have assumed that $ F $ is nowhere affine). 

\item If the solution of the auxiliary system of equations 
\eqref{eq-derivatives_system} has the form 
{\em (1.1)} or {\em (3.1)} in 
Theorem \ref{thm-sols_of_derivative_system}, then 
we can apply \cite[Theorem 16]{Kis22}. Indeed, 
the assumption of that theorem is fulfilled, because 
$ \psi_2 $ is not constant on any open subinterval 
while $ \varphi $ is a so-called 
{\em trigonometric fraction}. 

\item If the solution of the auxiliary system of equations 
\eqref{eq-derivatives_system} has the form 
{\em (1.2)} or {\em (3.2)} in 
Theorem \ref{thm-sols_of_derivative_system}, then 
we can apply \cite[Theorem 17]{Kis22}. 
The assumption of that theorem is indeed fulfilled, 
because $ \psi_2 $ is not constant on any open subinterval 
while $ \varphi $ is a so-called 
{\em linear fraction}. 

\item If the solution of the auxiliary system of equations 
\eqref{eq-derivatives_system} has the form 
{\em (1.3)} or {\em (3.3)} in 
Theorem \ref{thm-sols_of_derivative_system}, then 
we can apply \cite[Theorem 19]{Kis22}. 
The assumption of that theorem is indeed fulfilled, since 
$ \psi_2 $ is not constant on any open subinterval 
while $ \varphi $ is a so-called 
{\em hyperbolic fraction}. 
\end{itemize}
Theorems 16, 17 and 19 in the paper \cite{Kis22} provide 
necessary and sufficient conditions, therefore they 
provide the complete characterization of the solutions 
of \eqref{eq-Invariance-eq} in the case when 
$ F $ is nowhere affine. 
\end{remark}

\section{Open questions concerning our equation}\label{OQ}

As we have mentioned in the Introduction, our main 
equation \eqref{eq-Invariance-eq} is a reformulation of 
the invariance problem of 
generalized weighted quasi-arithmetic means. 
The functions $ F, f_k \,, G, g_k $ are constructed 
from the generator functions of the means in the invariance 
problem and their inverses. Those generators are assumed 
to be strictly monotone and continuous. 
Hence the ultimate goal would be to solve equation 
\eqref{eq-Invariance-eq} under the natural regularity 
assumptions, namely the continuity and strict monotonicity 
of the six unknown functions. However, at this point, \emph{it is still an open problem 
to determine the solutions of \eqref{eq-Invariance-eq} 
only assuming the differentiability of some (but not all) 
unknown functions}. 

In the derivation of our equation, its domain is originally given by the Cartesian product of the images of a given interval by two generator functions, which are not necessarily identical (for precise details, see the paper \cite{Kis22}.). In both the present paper and the earlier work \cite{Kis22}, we assumed that these images coincide; a condition that is, in fact, rather unnatural. 
\emph{Consequently, the question of what the solutions of the equation
\Eq{*}{
F\Big(\frac{x+y}2\Big)+f_1(x)+f_2(y)=G\big(g_1(x)+g_2(y)\big),\qquad x\in J_1\text{ and }y\in J_2
}
may be, in the case where the open intervals $J_1$ and $J_2$ are different, remains open.}


\section*{Statements and Declarations}

\subsection*{Funding}

Partial support for the first author's research was provided by NKFIH Grant K-134191 and by funding from the HUN-REN Hungarian Research Network. The research of the second author has been supported by 
the EK\"OP-24-0 University Research Scholarship Program of 
the Ministry for Culture and Innovation 
from the source of the 
National Research, Development and Innovation Fund; 
by the PhD Excellence Scholarship from the 
Count Istv\'an Tisza Foundation 
for the University of Debrecen; 
and by the University of Debrecen Program 
for Scientific Publication

\subsection*{Competing Interests}

The authors have no relevant financial 
or non-financial interests to disclose.

\subsection*{Author Contributions}
Both authors contributed equally to the conceptualization 
and execution of the research project. 
Section \ref{A} and \ref{PA} were mostly written by 
the first author while the second author contributed 
to these sections with observations and recommendations. 
Section \ref{NA} was mostly written by 
the second author while the first author contributed 
to this part with observations and recommendations. 
Sections \ref{Intro} and \ref{OQ} were edited jointly 
by the two authors. 
Both authors read and approved the final manuscript. 

\subsection*{Data Availability} 

No data was used for the research described in the article.

\end{document}